\documentclass[a4paper,12pt,reqno]{amsart}
\usepackage[T1]{fontenc}
\usepackage[utf8]{inputenc}
\usepackage{amssymb,bbm,mathrsfs}
\usepackage{amsmath}
\usepackage{amsthm}
\usepackage[margin=1in]{geometry}
\usepackage{enumitem}
\usepackage{xcolor}
\usepackage[bookmarksdepth=2]{hyperref}
\usepackage{graphicx}
\usepackage{tikz}
\definecolor{mygreen1}{RGB}{0,73,44} 
\definecolor{mygreen2}{RGB}{0,140,81}

\newtheorem{theorem}{Theorem}[section]
\newtheorem{lemma}[theorem]{Lemma}
\newtheorem{corollary}[theorem]{Corollary}
\newtheorem{proposition}[theorem]{Proposition}

\theoremstyle{definition}
\newtheorem{definition}[theorem]{Definition}

\newtheorem{remark}[theorem]{Remark}

\DeclareMathOperator{\re}{Re}
\DeclareMathOperator{\Res}{Res}
\DeclareMathOperator{\im}{Im}
\DeclareMathOperator{\sign}{sign}

\newcommand{\leb}{\mathscr{L}}

\newcommand{\fourier}{\mathscr{F}}

\newcommand{\laplace}{\mathscr{L}}

\newcommand{\pvint}{\operatorname{p{.}v{.}}\!\int}

\newcommand{\ph}{\varphi}
\newcommand{\eps}{\varepsilon}
\newcommand{\thet}{\vartheta}

\newcommand{\mdef}{\mathrel{\mathop:}=}

\newcommand{\Fpos}{F^+}
\newcommand{\Fneg}{F^-}
\newcommand{\Gpos}{G^+}
\newcommand{\Gneg}{G^-}

\renewcommand{\le}{\leqslant}
\renewcommand{\ge}{\geqslant}

\usepackage{ifthen}
\newcommand{\formula}[2][nolabel]%
{%
 \ifthenelse{\equal{#1}{nolabel}}%
 {\begin{align*} #2 \end{align*}}%
 {%
  \ifthenelse{\equal{#1}{}}%
  {\begin{align} #2 \end{align}}%
  {\begin{align} \label{#1} #2 \end{align}}%
 }%
}

\newcommand{\RR}{\mathbb{R}}
\newcommand{\EE}{\mathbb{E}}

\newcommand{\CC}{\mathbb{C}}

\newcommand{\Ha}{\mathcal{H}}

\theoremstyle{definition}

\usepackage{ifthen}

\newcommand{\C}{\mathbb{C}}
\newcommand{\PP}{\mathbb{P}}

\newcommand{\Arg}{\operatorname{Arg}}
\renewcommand{\Re}{\operatorname{Re}}
\renewcommand{\Im}{\operatorname{Im}}
\theoremstyle{definition}

\begin{document}
\title[Spectral theory]{Spectral theory for one-dimensional (non-symmetric) stable processes killed upon hitting the origin}
\author{Jacek Mucha}
\thanks{Work supported by National Science Centre (NCN), Poland, under grant 2015/19/B/ST1/01457}
\address{Jacek Mucha \\ Faculty of Pure and Applied Mathematics \\ Wrocław University of Science and Technology \\ ul. Wybrzeże Wyspiańskiego 27 \\ 50-370 Wrocław, Poland}
\email{jacek.mucha@pwr.edu.pl}
\date{\today}
\keywords{stable process, spectral theory}
%\subjclass[2010]{...}

\begin{abstract}
We obtain an integral formula for the distribution of the first hitting time of the origin for one-dimensional $\alpha$-stable processes $X_t$, where $\alpha\in(1,2)$. We also find a spectral-type integral formula for the transition operators $P_0^t$ of $X_t$ killed upon hitting the origin. Both expressions involve exponentially growing oscillating functions, which play a role of generalised eigenfunctions for $P_0^t$.
\end{abstract}

\maketitle

\section{Introduction}

The main purpose of this article is to extend the results of~\cite{symzero1}, where a large class of symmetric Lévy processes was considered, to non-symmetric stable Lévy processes. For such a process $X_t$, we study the hitting time of the origin, and transition operators of the process $X_t$ killed upon hitting the origin. We construct appropriate `generalised eigenfunctions' $\Fpos(s x)$ and $\Fneg(s x)$,  and we provide expressions similar to those of~\cite{symzero1}. In our case, however, the functions $\Fpos$ and $\Fneg$ are no longer bounded; in fact, they grow exponentially fast, and thus the methods of~\cite{symzero1} need to be substantially modified. Our approach is similar to that of~\cite{specHL}, where similar problems for hitting a half-line are studied. However, we avoid the use of special functions. Instead, we consider integral-type expressions, related to some extend to~\cite{rogers3}; see also the prelimiary version~\cite{rogers} of that article.

Hitting times for Markov processes are one of the fundamental objects in probabilistic potential theory, with numerous applications to other areas of mathematics. Various results for the Brownian motion are collected in~\cite{borodin}, Appendix~1. Hitting times for symmetric stable Lévy processes have already been studied in 1960s, see, for example, \cite{blumenthal}. In~\cite{doney} a formula for the density of the hitting time of the origin was obtained for spectrally positive L\'{e}vy processes. The theory was further developed in~\cite{pakes}, and in~\cite{peskir} a series expansion of the density of the hitting time of the origin was found for stable processes with no negative jumps. Further results for completely asymmetric stable processes were presented in~\cite{simon}, where series representations for the density are presented, and in Section 46 of~\cite{sato}. The Mellin transform for the hitting time of zero is given in~\cite{kuznetsov}. Additionally, in~\cite{simon}, hitting times of points are proved to be unimodal when $\alpha\leqslant 3/2$. Later, in~\cite{letemplier}, unimodality was proven for $\alpha\in (1,2]$. More general results about unimodality of hitting times for Markov processes can be found in~\cite{rosler}.  Asymptotic analysis of the hitting times of points can also be found in~\cite{kesten,port,yano}. Estimates for hitting times of points for more general symmetric L\'{e}vy processes were obtained under some mild regularity assumptions in~\cite{grzywny}.

Obviously, hitting times and distributions for stable processes have been studied also for more general sets. These results are, however, of much different nature, and we only mention some examples that are at least remotely related to our work. Hitting distributions of the interval $[-1,1]$ or its complement $\RR \setminus (-1, 1)$ have been found in~\cite{kyprianou1} and~\cite{rogozin}, respectively; see also~\cite{profeta:simon} for further discussion and references. Hitting times of half-lines, called \emph{first passage times}, are of particular interest, being the main subject of fluctuation theory for Lévy processes; we mention here~\cite{graczyk:jakubowski,kuznetsov:extrema,specHL}.

As mentioned above, spectral theory of symmetric L\'{e}vy processes killed upon hitting $\{ 0 \}$ is developed in~\cite{symzero1}. Further work in this are can be found in~\cite{tomek}, where a narrower class of symmetric L\'evy processes with completely monotone jumps is studied. Similar work for symmetric processes in half-line can be found in~\cite{bm} and~\cite{kmr}, which extend the former work~\cite{cauchy} on the Cauchy process. Non-symmetric stable processes in half-line have been studied in a similar way in~\cite{specHL}; see also~\cite{rogers,rogers3} for preliminary results for more general non-symmetric Lévy processes with completely monotone jumps. We note that spectral theory of non-symmetric Markov processes on the half-line was also studied in~\cite{mandl} (one-dimensional diffusion processes), \cite{ogura} (branching processes) and~\cite{patie} (non-self-adjoint Markov semigroups).

Our main goal is to extend the results from~\cite{symzero1} to the class of non-symmetric $\alpha$-stable processes, $\alpha\in(1,2)$. The symmetric case is much easier, mainly because in this case the characteristic exponent of the process is real. This property is crucial for the method developed in~\cite{symzero1}. However, the tools developed in~\cite{specHL} allow to follow some of the arguments from~\cite{symzero1}, after appropriate deformation of the contour of integration to the line along which characteristic exponent takes real values.

Let us briefly motivate the form of our main result, Theorem~\ref{thm:3}. If $\tau$ is the hitting time of the complement of a compact set $D$ for a sufficiently regular symmetric Markov process $X_t$, then the transition operators $P^D_t$ of the process $X_t$ killed at $\tau$ are compact operators on $L^2(D)$, and it is easy to find spectral expansion of $\PP^x(\tau > t)$: we have
\[
 \PP^x(\tau>t) = \sum_{n = 1}^\infty e^{-\lambda_n t} \ph_n(x) \int_D \ph_n(y) dy ,
\]
where $(\ph_n : n = 1, 2, \ldots)$ is a complete orthonormal system of eigenfunctions of $P^D_t$, with corresponding eigenvalues $e^{-\lambda_n t}$. We refer to~\cite{getoor} for a rigorous discussion of an analogous description of the transition density (the kernel of $P^D_t$). When $D$ is unbounded and $P^D_t$ fail to be compact operators, one can expect that a continuous variant of the above expansion holds:
\[
 \PP^x(\tau>t) = \int_S e^{-\lambda(s) t} \ph_s(x) \biggl(\int_D \ph_s(y) dy \biggr) m(ds) ,
\]
where $\ph_s$ are \emph{generalised eigenfunctions} (or \emph{resonanses}) of $P^D_t$ with \emph{generalised eigenvalues} $\lambda(s)$. Here $S$ is some parameter set, and $m$ is an appropriate meusure on $S$; again we refer to~\cite{getoor} for a rigorous discussion of such expansion for transition densities. Similar problem for non-symmetric processes are generally much harder. However, in certain cases one can hope for similar expansions. In the compact case, if $P^D_t$ admit a complete system of eigenfunctions $\ph_n^-$ and co-eigenfunctions $\ph_n^+$, then it is expected that
\[
 \PP^x(\tau>t) = \sum_{n = 1}^\infty e^{-\lambda_n t} \ph_n^-(x) \int_D \ph_n^+(y) dy .
\]
Similar expressions are possible in the non-compact case, of the form
\[
 \PP^x(\tau>t) = \int_S e^{-\lambda(s) t} \ph_s^-(x) \biggl(\int_D \ph_s^+(y) dy \biggr) m(ds) .
\]
The article~\cite{specHL} derives a formula of the form given above for the first exit time from $(0,\infty)$ for non-symmetric $\alpha$-stable processes. Here we prove an analogous result for the hitting time of $0$. In either case the generalised eigenfunctions $\ph_s^-(x) = F(s x)$ have exponential growth; namely, we (roughly) have $F(x)=e^{a x}\sin(b x + c) - G(x)$ for a reasonably small remainder term $G$. As we shall see below, this rapid growth of $F$ is a constant source of problems in applications of Fubini's theorem, invertions of Laplace transforms etc.

Much of the inspiration for the present work also came from the theory of \emph{Rogers functions}~(\cite{rogers,rogers3}). The characteristic exponent of a stable process is a particularly simple example of a Rogers function. Many of the results presented below seem to extend to more general Rogers functions, which suggests that our main results can possibly be extended to more general Lévy processes with completely monotone jumps.

\subsection{Main results}
Let $X_t$ be the $\alpha$-stable process with index $\alpha \in (1, 2)$ and positivity parameter $\rho \in [1 - \tfrac{1}{\alpha}, \tfrac{1}{\alpha}]$. We assume that $X_t$ is normalised in such a way that if $\psi$ is the characteristic exponent of $X_t$, then $|\psi(1)| = 1$. Let
\[
 \theta = (1 - 2 \rho) \tfrac{\pi}{2} ,
\]
so that $\psi(\xi) = e^{-i \alpha \theta} |\xi|^\alpha$ for $\xi > 0$ and $\psi(\xi) = e^{i \alpha \theta} |\xi|^\alpha$ for $\xi < 0$. We remark that our assumption $\alpha \in (1, 2)$ is not restrictive: if $\alpha \leqslant 1$, then, with probability one, $X_t$ never hits $0$.

Define the functions $\Gpos$ and $\Gneg$ by the formulas
\begin{equation}
\label{eq:gdef}
\begin{aligned}
\Gpos(x) = \Gneg(-x) & = \frac{\alpha \sin \tfrac{\pi}{\alpha}}{\pi} \int_0^\infty \frac{t^\alpha \sin(\alpha (\tfrac{\pi}{2} - \theta))}{t^{2 \alpha} - 2 t^\alpha \cos(\alpha (\tfrac{\pi}{2} - \theta)) + 1} \, e^{-t x} dt , \\
\Gpos(-x) = \Gneg(x) & = \frac{\alpha \sin \tfrac{\pi}{\alpha}}{\pi} \int_0^\infty \frac{t^\alpha \sin(\alpha (\tfrac{\pi}{2} + \theta))}{t^{2 \alpha} - 2 t^\alpha \cos(\alpha (\tfrac{\pi}{2} + \theta)) + 1} \, e^{-t x} dt
\end{aligned}
\end{equation}
for $x > 0$. Define furthermore
\[
\Fpos(x) = \Fneg(-x) = e^{-x \sin\theta} \sin\bigl(|x| \cos\theta + \theta \sign x + \tfrac{\pi}{\alpha} - \tfrac{\pi}{2} \bigr) - \Gpos(x) .
\]
Note that $\Gneg$ and $\Fneg$ are given by the same expressions as $\Gpos$ and $\Fpos$, with $\rho$ replaced by $1 - \rho$ (that is, with $\theta$ changed to $-\theta$). By $\PP^x$ and $\EE^x$ we denote the probability and expectation corresponding to the process $X_t$ started at $x$. The following theorem is the first main result of the paper.

\begin{theorem} 
\label{thm:3} Let $\tau_0$ be the first hitting time of $\{ 0 \}$ for the process $X_t$. Then
\begin{equation}
\label{formula:1}
\PP^x(\tau_0 > t) = \frac{1}{\pi\cos\theta}\int_0^\infty \frac{e^{-s^\alpha t}}{s} \, \Fneg (sx)ds.
\end{equation}
for $x \ne 0$ and $t>0$.
\end{theorem}

The functions $\Fpos$, $\Fneg$ can be seen as generalised eigenfunctions of transition operators of $X_t$ killed upon hitting $\{ 0\}$. These operators are defined by the formula
\begin{equation}
P_t^{\RR\setminus \{ 0 \}}f(x)=\EE^x(f(X_t); t < \tau_0),
\end{equation}
for $t>0$, $x\in\RR\setminus \{ 0 \}$, and they act on $\leb^p(\RR\setminus \{ 0 \} )$ for arbitrary $p\in [1,\infty]$. Although $P_t^{\RR \setminus \{0\}} \Fneg$ is not well-defined (the expectation does not converge), we have the following spectral-type representation of $P_t^{\RR \setminus \{0\}}$. This is our second main result.

\begin{theorem}
\label{thm:1}
There is a class of functions $\Ha$, dense in $\leb^2(\RR \setminus \{0\})$, with the following property. If $f,g\in\Ha$, then
\[
\begin{split}
& \int_{-\infty}^\infty  P_t^0 f(x)g(x)dx  =  \int_0^\infty \frac{ e^{-s^\alpha t}}{\cos\theta} \left(\int_{-\infty}^\infty \Fpos(sx)f(x)dx \right) \left( \int_{-\infty}^\infty \Fneg(sy)g(y)dy \right) ds\\
&\quad + \int_0^\infty \frac{e^{-s^\alpha t}}{ \cos\theta}\left( \int_{-\infty}^\infty e^{-sx\sin\theta} \sin(sx\cos\theta)f(x)dx\right) \left( \int_{-\infty}^\infty e^{sy\sin\theta}\sin(sy\cos\theta)g(y)dy\right) ds.
\end{split}
\]
\end{theorem}

The class $\Ha$ is discussed in detail in Section~\ref{sec:test}. Theorem~\ref{thm:1} provides a spectral-type representation of $P_t^0$: the parenthesised integrals can be thought of as Fourier-type transforms of $f$ and $g$, which diagonalise the action of $P_t^0$. This is the reason we call $\Fpos$ and $\Fneg$ \emph{generalised eigenfunctions} of $P_t^0$. We stress, however, that due to exponential growth at infinity, $P_t^0 \Fpos$ is not defined unless the process $X_t$ is symmetric, that is, $\theta = 0$.

We mention here one property of the functions $\Gpos$ and $\Gneg$. For further information, see Section~\ref{sec:g}.

\begin{proposition} 
\label{pr1}
The functions $\Gpos$ and $\Gneg$ are bounded, integrable, and their Fourier transform is given by
\[
 \fourier \Gpos(\xi) = \fourier \Gneg(-\xi) = \sin \tfrac{\pi}{\alpha} \biggl(\frac{\alpha}{\psi(\xi)-1} - \frac{1}{e^{-i \theta} \xi - 1} + \frac{1}{e^{i \theta} \xi + 1} \biggr), \qquad \xi \in \RR;
\]
Furthermore, the functions $\Gpos(x) = \Gneg(-x)$ and $\Gpos(-x) = \Gneg(x)$ are completely monotone on $(0, \infty)$.
\end{proposition}

\subsection{Structure of the article}

The remainning part of the paper is divided into four sections. In Preliminaries we recall basic definitions and state auxiliary lemmas. In particular, we introduce a suitable family of test functions $\Ha$, and we discuss Nevanlinna class of functions and Cauchy's integral formula. In Section 3 we prove a handful of technical lemmas in order to derive a formula for the generalised eigenfunctions $\Fpos$, $\Fneg$. The properties of functions $\Gpos$ and $\Gneg$ are studied here, and the proof of Proposition~\ref{pr1} is given. Section 4 is dedicated to the proof of the Theorem~\ref{thm:1} and, finally, Theorem~\ref{thm:3} is proved in Section 5.

In our proof, we deform the contour of integration a number of times. Here is a rough sketch of the argument.
\begin{itemize}
\item We begin with the triple integral $I(\lambda)$ of $e^{-\lambda t} p^0_t(x, y) f(x) g(y)$.
\item We use Plancherel's theorem to rewrite $I(\lambda)$ as a triple integral with respect to $t$, $\xi$, $\eta$, where $\xi$ and $\eta$ are the Fourier variables corresponding to $x$ and $y$ (Lemma~\ref{lem:fgp}).
\item Next, we deform the contour of integration in $\xi$ and $\eta$ to $(-e^{-i \theta} \infty, 0) \cup (0, e^{i \theta} \infty)$.
\item By doing so, we obtain an expression for $I(\lambda)$, which is a Cauchy--Stieltjes transform $\phi_4(\lambda)$ of some function of a new variable $r$ (Lemma~\ref{lem:fgp2}).
\item Since the Cauchy--Stieltjes transform ($r \mapsto \lambda$) is the Laplace transform ($t \mapsto \lambda$) of the Laplace transform ($r \mapsto t$), the above leads to an expression for the double integral $J(t)$ of $p^0_t(x, y) f(x) g(y)$ (with respect to $x$ and $y$) as a Laplace transform of what will be denoted by $-\im \phi_4(-r)$ (Theorem~\ref{thm:imagine}).
\item In order to prove Theorem~\ref{thm:1}, we now identify the expression for $-\im \phi_4(-r)$ (which is given in terms of integrals of Laplace transforms of $f$ and $g$) with an appropriate integral transform of $f$ and $g$. This involves deforming back the contour of integration with respect to $\xi$ and $\eta$ to $\RR$, and an application of Plancherel's formula. It is here convenient to replace $r$ with $s^\alpha$.
\item Theorem~\ref{thm:3} is proved in a similar way, with an additional step at the end of the proof: we change the order of the integrals with respect to $s$ and $x$, and by a density argument, we are able to remove the integral with respect to $x$. Changing the order of integration, however, is not straightforward: it requires an appropriated deformation of the contour of integration, so that Fubini's theorem can be applied.
\end{itemize}

\section{Preliminaries}

We denote by $\leb^p(\RR)$ the space of real-valued functions $f$ on $\RR$ such that $|f(x)|^p$ is integrable. We use $\laplace f$ to denote the two-sided Laplace transform of $f$:
\[ \laplace f(\xi) = \int_{-\infty}^\infty f(x) e^{-\xi x} dx \]
whenever the integral converges absolutely. If $f \in \leb^1(\RR)$, then $\fourier f(\xi) = \laplace f(i \xi)$ (with $\xi \in \RR)$ is the Fourier transform of $f$. The Fourier transformation $\fourier$ is extended continuously to $\leb^2(\RR)$.

\subsection{Stable L\'evy processes}
By $X_t$ we denote a one-dimensional $\alpha$-stable Lévy process with index of stability $\alpha \in (1, 2)$. We assume that $\alpha > 1$ in order that $X_t$ is point-recurrent (i.e.\@ it hits single points with positive probability). The case $\alpha = 2$ is well-studied and much simpler, so we require that $\alpha \ne 2$.

A one-dimensional stable Lévy process is completely characterised by $\alpha$, the positivity parameter $\rho = \PP^0(X_1 > 0)$, and the scale parameter $k > 0$. For $\alpha \in (1, 2)$, we have $\rho \in [1 - \tfrac{1}{\alpha}, \tfrac{1}{\alpha}]$. We denote by $\psi$ the characteristic exponent of $X_t$:
\begin{equation}
\label{eq:lk}
\EE^0 e^{i\xi X_t}=e^{-t\psi(\xi)}, \ t>0,\xi\in\RR.
\end{equation}
In our case
\[
\psi(\xi) = (k|\xi|)^\alpha\left(1-i \tan\left( (2\rho-1)\frac{\alpha\pi}{2}\right) \sign \xi\right) , \qquad \xi \in \RR.
\]
Our results do not depend on the scale parameter $k$ in any essential way. For this reason, we choose $k$ in such a way that $|\psi(\xi)| = |\xi|^\alpha$ for every $\xi \in \RR$. Thus, if we set
\[
\theta = \frac{(1 - 2 \rho) \pi}{2} \, ,
\]
then we have $|\theta| \le \tfrac{\pi}{\alpha} - \tfrac{\pi}{2}$ and
\[
\psi(\xi) = \begin{cases} (e^{-i \thet} \xi)^\alpha & \text{if $\xi > 0$,} \\ (e^{i \thet} (-\xi))^\alpha & \text{if $\xi < 0$} \end{cases}
\]
(all complex powers are principal branches). Note that if we replace conditions $\xi > 0$, $\xi < 0$ by $\re \xi > 0$, $\re \xi < 0$, respectively, then the above expression defines a holomorphic extension of $\psi$ to $\C \setminus i \RR$. Throughout the text, the symbol $\psi$ denotes this extension, and the fact that $\psi(r e^{i \thet}) \in (0, \infty)$ for $r \in (0, \infty)$ will play an important role.

\begin{remark}
Apparently, the results of the present article can be extended to some Lévy processes with completely monotone jumps, introduced in~\cite{rogers2} and studied recently in~\cite{rogers3}. For \emph{symmetric} Lévy processes this was already done in~\cite{symzero1}. In the non-symmetric case, one clearly has to assume that $1 / (1 + \psi(\xi))$ is absolutely integrable, so that $X_t$ is point-recurrent, and points are regular for $X_t$ (see~\cite{sato}, Theorem~43.3). However, a number of further technical conditions will have to be imposed.
\end{remark}

Probability and expectation of the process starting from $x\in\RR$ are denoted by $\PP^x$ and $\EE^x$. We define the transition operators $P_t$ of $X_t$ by
\[
P_tf(x)=\EE^xf(X_t)=\int_\RR f(y)\PP^x(X_t\in dy), \qquad t>0, x\in\RR.
\]
The operators $P_t$ are convolution operators, and the corresponding convolution kernels $p_t(x)$ are known as transition densities. The operators $P_t$ form a strongly continuous semigroup of operators on $\leb^2(\RR)$, and their action is diagonalised by the Fourier transformation.

Let $D\subseteq \RR$ be an open set and let
\[
\tau_D=\inf\{t\geqslant 0: X_t\notin D\}
\]
be the first exit time of the process $X_t$ from the set $D$. The process $X_t$ killed upon leaving $D$ is formally defined to be equal to $X_t$ until its life-time $\tau_D$. We are more interested in the corresponding transition operators $P_t^D$, given by
\[
P_t^Df(x)= \EE^x (f(X_t)\chi_{t<\tau_D})=\int_Df(y)\PP^x(X_t\in dy; t<\tau_D), \qquad t>0,x\in D.
\]
The corresponding kernel function $p^D_t(x, y)$, the transition density of the killed process, is given by Hunt's switching formula
\[
p^D_t(x, y) = p_t(y - x) - \EE^x(p_{t - \tau_D}(y - X_{\tau_D}) \chi_{t < \tau_D}) , \qquad t > 0, x, y \in D.
\]
In particular, $0 \leqslant p^D_t(x, y) \leqslant p_t(y - x)$.We consider $D=\RR\setminus \{ 0 \}$, and for simplicity we denote $\tau_0 = \tau_{\RR \setminus \{0\}}$, $P_t^0 = P_t^{\RR\setminus \{ 0 \}}$ and $p^0_t(x,y) = p^{\RR \setminus \{0\}}_t(x,y)$.

Recall that the $\lambda$-potential kernel of the process $X_t$ is defined as
\[
u_\lambda(x)=\int_0^\infty e^{-\lambda t}p_t(x)dt, \qquad \lambda > 0, x\in\RR.
\]
By~\eqref{eq:lk} and Fubini's theorem, we have
\[
\fourier p_t(\xi) = e^{-t \psi(-\xi)} , \qquad \fourier u_\lambda(\xi) = \frac{1}{\psi(-\xi) + \lambda} \, , \qquad \xi \in \RR.
\]
In a similar way, the $\lambda$-potential kernel of the killed process is given by
\[
u^0_\lambda(x,y)=\int_0^\infty e^{-\lambda t}p^0_t(x,y)dt, \qquad \lambda > 0, x,y\in\RR.
\]
By Proposition 41.3 in~\cite{sato}, we have
\begin{equation}\label{eq:uzero}
u_\lambda^0(x,y)=u_\lambda(y-x)-\frac{u_\lambda(-x)u_\lambda(y)}{u_\lambda(0)}.
\end{equation}
Furthermore,
\[
\int_{-\infty}^\infty u_\lambda^0(x,y)dy=\int_0^\infty e^{-\lambda t} \PP^x(\tau_0>t)dt=\frac{1}{\lambda} \, (1 - \EE^x e^{-\lambda \tau_0}) ,
\]
and, on the other hand,
\[
\int_{-\infty}^\infty u_\lambda^0(x,y)dy=\int_{-\infty}^\infty u_\lambda(y-x)dy-\frac{u_\lambda(-x)}{u_\lambda(0)}\int_{-\infty}^\infty u_\lambda(y)dy=\frac{1}{\lambda}\left( 1-\frac{u_\lambda(-x)}{u_\lambda(0)}\right).
\]
It follows that
\begin{equation}
\label{expected}
\EE^x e^{-\lambda \tau_0}= \frac{u_\lambda(-x)}{u_\lambda(0)}.
\end{equation}

Our ultimate goal is to find integral expressions for the distribution of $\tau_0$ with respect to $\PP^x$. Our starting points are~\eqref{eq:uzero} and~\eqref{expected}, which describe the Laplace transforms (with respect to $t$) of $p^0_t(x, y)$ and $\PP^x(\tau_0 \in dt)$ in terms of $u_\lambda$. The function $u_\lambda$ is in turn the inverse Fourier transform of $1 / (\lambda + \psi(-\xi))$. These expressions, however, are not suitable for standard inversion formulas. For this reason, we will first multiply the above expressions by appropriately regular test functions, and only then consider Fourier transforms in $x$ and $y$.

%%%%%%%%%%%%%%%%%%%%%%%%%%%%%%%%%%%%%%%%%%%%%%%%%%%%%%%%%%%%%%%%%%%%%%%%%%%%%%%%%%%%%%%%%%
\subsection{Test functions}
\label{sec:test}
%%%%%%%%%%%%%%%%%%%%%%%%%%%%%%%%%%%%%%%%%%%%%%%%%%%%%%%%%%%%%%%%%%%%%%%%%%%%%%%%%%%%%%%%%%%%%

Just as in the symmetric case studied in~\cite{symzero1}, our representation of $p^0_t(x,y)$ and $\PP^x(\tau_0 > t)$ involves \emph{generalised eigenfunctions} $\Fpos$ and $\Fneg$. In the symmetric case, $\Fpos = \Fneg$ is a bounded function; here $\Fpos$ and $\Fneg$ have exponential growth at infinity. This nuisance makes the use of the Laplace transform problematic. To overcome this difficulty, we introduce a particular class of test functions, following~\cite{specHL} (where stable processes in a half-line were studied).

\begin{definition} 
By $\Ha_+$ denote the class of all functions $g: \RR \to \RR$ such that
\begin{enumerate}
\item $g(x) = 0$ for $x < 0$;
\item $g$ extends to an analytic function in the sector $|\arg(z)|<\pi/2$;
\item for every $\varepsilon\in(0,\pi/2)$ there exists $\delta=\delta(\varepsilon)>0$ such that $|g(z)|=O(|z|^{-\delta |z|})$ when $|z|\rightarrow\infty$, and $|g(z)|=O(1)$ when $|z|\rightarrow 0$, uniformly in the sector $|\arg(z)|<\pi/2-\varepsilon$.
\end{enumerate}
We say that $g \in \Ha_-$ if $g(-z)$ belongs to $\Ha_-$. Finally, $g\in \Ha$ if $g : \RR \to \RR$, and $g = g_+ + g_-$ for some $g_+ \in \Ha_+$ and $g_- \in \Ha_-$.
\end{definition}

We note that the classes $\Ha_+$, $\Ha_-$, $\Ha$ are non-empty and non-trivial, e.g.\@ $g(z)=e^{-s|z|\log|z + e|}$ is in $\Ha$ for every $s > 0$. It was observed in~\cite{specHL} that, on one hand, $\Ha_+$ is sufficiently rich, while on the other one, the Laplace--Fourier transform of a function $g\in\Ha_+$ is a suitably decaying analytical function in the sector $|\Arg(z)|<\pi-\varepsilon$ for every $\varepsilon > 0$. We rephrase these results for the class $\Ha$.

\begin{lemma}[\cite{specHL}, p.~19, Lemma~2.14]
\label{lm:7}
Let $g\in\Ha$. Then $\laplace g(z)$ is an entire function and for every $\varepsilon \in (0, \frac{\pi}{2})$ there exists a constant $C$ such that
\[
|\laplace g(z)|\leqslant C\min\{1,|z|^{-1}\}, \qquad \lvert\arg(iz)\rvert \leqslant \tfrac{\pi}{2}-\varepsilon \text{ or } \lvert\arg(-iz)\rvert \leqslant \tfrac{\pi}{2}-\varepsilon.
\]
\end{lemma}

A similar argument leads to the following estimate of the derivative of $\laplace g$; we omit the proof.

\begin{lemma}
\label{lm:7p}
Let $g\in\Ha$. Then for every $\varepsilon \in (0, \frac{\pi}{2})$ there exists a constant $C$ such that
\[
|(\laplace g)'(z)|\leqslant C\min\{1,|z|^{-2}\}, \qquad \lvert\arg(iz)\rvert \leqslant \tfrac{\pi}{2}-\varepsilon \text{ or } \lvert\arg(-iz)\rvert \leqslant \tfrac{\pi}{2}-\varepsilon.
\]
\end{lemma}

\begin{lemma}
\label{lm:density}

The class $\Ha$ is dense in the following sense: if $f$ is a Borel-measurable function and
\[
 \int_{-\infty}^\infty f(x) g(x) dx = 0 \qquad \text{for every $g \in \Ha$,}
\]
with the integral absolutely convergent, then $f(x) = 0$ for almost every $x \in \RR$.
\end{lemma}

\begin{proof}
Let $\rho(x)=x\log(x + e)$  and consider $g(x)=e^{-s\rho(x)}\chi_{(0,\infty)}(x)$ for some $s>0$. As we already remarked above, $g\in\Ha$: it extends to an analytic function in the sector $|\arg(z)|<\tfrac{\pi}{2}$, and we have $|g(z)|=|z+e|^{-s|z|}=O(|z|^{-s|z|/2})$ as $|z|\rightarrow\infty$ and $|g(z)|=O(1)$ as $|z| \to 0$, uniformly in the sector $|\arg z| \le \tfrac{\pi}{2} - \epsilon$ for every $\epsilon > 0$. Note that $\rho^\prime(x)>0$ and, by substitution $y=\rho(x)$, we get
\begin{equation}
0 = \int_{-\infty}^\infty f(x) g(x) dx = \int_0^\infty e^{-sy}f(\rho^{-1}(y))(\rho^{-1})^\prime(y)dy 
\end{equation}
for every $s>0$. This means that the Laplace transform of $f(\rho^{-1}(y))(\rho^{-1})^\prime(y)$ vanishes for every $s>0$, and hence $f(\rho^{-1}(y))(\rho^{-1})^\prime(y)=0$ for almost every $y > 0$. It follows that $f(x)=0$ for almost every $x > 0$. A very similar argument shows that $f(x) = 0$ for almost every $x < 0$, and the proof is complete.
\end{proof}

%%%%%%%%%%%%%%%%%%%%%%%%%%%%%%%%%%%%%%%%%%%%%%%%%%%%%%%%%%%%%%%%%%%%%%%%%%%%%%%%%%%%%%%%

\subsection{Stieltjes functions and Cauchy's integral formula for the upper half-plane}

One of the key steps in the proof of our main result involves Stieltjes-type representation of analytic functions in $\C \setminus (-\infty, 0)$. We deduce this result from Cauchy's integral formula for functions from the Hardy space $\mathscr{H}^p$ in the upper complex half-plane $H_+ = \{z \in \C : \im z > 0\}$. We begin with a number of standard definitions.

\begin{definition}
A function $f : (0, \infty) \to \RR$ is a \emph{Stieltjes function} if
\begin{equation}
\label{rep:st}
f(x)=\frac{c_1}{x}+c_2+\frac{1}{\pi}\int_0^\infty \frac{m(ds)}{x+s}, \qquad x>0,
\end{equation}
where $c_1,c_2\geqslant 0$ and $m$ is a non-negative Radon measure on $(0,\infty)$ which satisfies the integrability condition $\int_0^\infty \min\{1,s^{-1}\}m(ds)<\infty$.
\end{definition}

\begin{definition}
For $p \in (1, \infty)$, by $\mathscr{H}^p$ we denote the space of functions $f$ analytic in $H_+$ such that $|f(x+iy)|^p$ is integrable with respect to $x \in \RR$ for each $y>0$, and $\int_\RR |f(x+iy)|^p dx$ is a bounded function of $y > 0$.
\end{definition}

\begin{definition}
A function $f$ analytic in $H_+$ is an \emph{outer function} if
\begin{equation}
\log |f(x + i y)|=\frac{1}{\pi} \int_0^\infty \frac{y}{(x-s)^2+y^2}\log|f(s)|ds
\end{equation}
for every $x \in \RR$ and $y > 0$. Here for $s \in \RR$ the symbol $f(s)$ denotes the limit $\lim_{t \to 0^+} f(s + i t)$, which necessarily exists for almost every $s \in \RR$.
\end{definition}

\begin{definition}
The Nevanlinna class $\mathscr{N}^+$ is the set of functions $f$ analytic in $H_+$, which can be represented as $f=f_1/f_2$, where $f_1$ and $f_2$ are analytic and bounded in $H_+$, and $f_2$ is outer.
\end{definition}

The class of outer functions on the unit disk, and the Nevanlinna class $\mathscr{N}^+$ on the unit disk, are defined in a similar way. Note that composition with a conformal map between the unit disk and the upper complex half-plane defines a bijection between the corresponding classes of outer functions, as well as between the corresponding Nevanlinna classes $\mathscr{N}^+$. We refer to~\cite{branges,rosenblum} for further details.

We will use the following standard properties of Stieltjes functions, Hardy space $\mathscr{H}^p$, and the Nevanlinna class $\mathscr{N}^+$.

\begin{lemma}[\cite{symzero1}, Proposition~2.1(b)]
\label{lm:stieltjes}
A function $f$ is a Stieltjes function if and only in $f(z)\geqslant 0$ for $z>0$ and either $f$ is constant, or $f$ extends to a holomorphic function in $\CC\setminus (\infty,0]$, which swaps the upper and the lower complex half-planes, i.e.\@ if $\Im z>0$, then $\Im f(z)<0$ and if $\Im z<0$, then $\Im f(z)>0$.
\end{lemma}

\begin{lemma}[\cite{duren}, Theorem~3.2; \cite{rosenblum}, Theorem~4.29]
\label{lm:BTcircle2}
If $f$ is an analytic function in the unit disk with non-negative real part, then $f$ is an outer function.
\end{lemma}

\begin{lemma}
\label{lm:BTcircle}
If $f$ is an analytic function in $H_+$ and $\re f \ge 0$ in $H_+$, then $f$ is an outer function. If $f$ is a Stieltjes function, then (the analytic extension of) $f$ is an outer function.
\end{lemma}

\begin{proof}
Suppose that $\re f \ge 0$ in $H_+$, and consider the conformal map $w(z)=-i (z-1) / (z+1)$ from the unit disk $\{z \in \CC: |z|<1\}$ onto the upper complex half-plane $H_+$. If $g(z) = f(w(z))$ for $z \in B_1$, then $g$ is an analytic function in the unit disk, with non-negative real part. By Lemma~\ref{lm:BTcircle2}, $g$ is an outer function on the unit disk. It follows that $f(z) = -i g(w^{-1}(z))$ is an outer function in the upper complex half-plane $H_+$.

If $f$ is the analytic extension to the upper complex half-plane $H_+$ of a Stieltjes function, then, by Lemma~\ref{lm:stieltjes}, $\re (i f(z)) \ge 0$ for every $z \in H_+$. By the first part of the proof, $i f$ is outer, and hence also $f$ is outer.
\end{proof}

\begin{lemma}[\cite{rosenblum}, 5.14(iv)]
\label{lm:sum.is.N}
The sum of functions from $\mathscr{N}^+$ is in $\mathscr{N}^+$. The product of functions from $\mathscr{N}^+$ is in $\mathscr{N}^+$. The ratio of a function in $\mathscr{N}^+$ and an outer function is in $\mathscr{N}^+$.
\end{lemma}

\begin{lemma}
\label{lm:id.is.N}
The function $f(z) = \sqrt{z}$ is outer.
\end{lemma}

\begin{proof}
The function $1 / f(z) = 1 / \sqrt{z}$ is a Stieltjes function. Hence, $1 / f$ is outer, and it follows that $f$ is outer, too.
\end{proof}

\begin{lemma}[\cite{rosenblum}, Theorems~5.6 and~5.23(i)] 
\label{lm:Np.to.H2}
Let $f$ be in the Nevanlinna class $\mathscr{N}^+$. Then the boundary limit $f(x) = \lim_{y \to 0^+} f(x + i y)$ exists for almost every $x \in \RR$. Furthermore, if
\begin{equation}
\int_\RR |f(x)|^p dx <\infty
\end{equation}
for some $p \in (1, \infty)$, then $f\in\mathscr{H}^p$.
\end{lemma}

\begin{theorem}[\cite{duren}, Theorem~11.8] 
\label{thm:branges}
Let $f$ be in $\mathscr{H}^p$ for some $p \in (1, \infty)$. Then
\begin{equation}
f(z)=\frac{1}{2\pi i }\int_\RR \frac{f(t)}{t-z}dt
\end{equation}
when $\Im z>0$ and
\begin{equation}
0=\frac{1}{2\pi i }\int_\RR \frac{f(t)}{t-z}dt
\end{equation}
when $\Im z<0$.
\end{theorem} 

\begin{corollary}
\label{cor:branges}
Let $f$ be an analytic function in $\C \setminus (-\infty, 0]$, which is real-valued on $(0, \infty)$, and such that $f(-\xi^2)$ belongs to the Nevanlinna class $\mathscr{N}^+$ in the upper complex half-plane. For $s > 0$ denote by $f(-s)$ the boundary limit $\lim_{t \to 0^+} f(-s + i t)$ (which exists for almost every $s$). Suppose that $\int_0^\infty |f(-s^2)|^p ds < \infty$ for some $p \in (1, \infty)$. Then
\begin{equation}
f(z) = -\frac{1}{\pi}\int_0^\infty \frac{\im f(-s)}{s + z} \,ds = \frac{1}{\pi} \int_0^\infty \frac{\sqrt{z}}{\sqrt{s}} \, \frac{\re f(-s)}{s + z} \, ds , \qquad z \in \C \setminus (-\infty, 0] .
\end{equation}
\end{corollary}

\begin{proof}
We define an auxiliary function $g(\xi) = f((-i \xi)^2)$ in the upper complex half-plane $\im \xi > 0$, and we verify that $g$ satisfies the assumptions of Theorem~\ref{thm:branges}. Since $f(\overline{z}) = \overline{f(z)}$, $g$ has a boundary limit almost everywhere, given by $g(t) = \overline{f(-t^2)}$ and $g(-t) = f(-t^2)$ for $t \geqslant 0$. By assumption, $g$ is in the Nevanlinna class $\mathscr{N}^+$. By Lemma~\ref{lm:Np.to.H2}, $g$ is in the Hardy space $\mathscr{H}^p$. Therefore, by Theorem~\ref{thm:branges}, for $z \in \C \setminus (-\infty, 0]$ we have
\[
f(z) = g(i \sqrt{z}) = \frac{1}{2\pi i} \int_{-\infty}^\infty \frac{g(t)}{t-i \sqrt{z}}dt = \frac{1}{2 \pi i} \int_0^\infty \frac{\overline{f(-t^2)}}{t - i \sqrt{z}} \, dt - \frac{1}{2 \pi i} \int_0^\infty \frac{f(-t^2)}{t + i \sqrt{z}} \, dt ,
\]
and, similarly,
\[
0 = \frac{1}{2\pi i} \int_{-\infty}^\infty \frac{g(t)}{t + i \sqrt{z}}dt = \frac{1}{2 \pi i} \int_0^\infty \frac{\overline{f(-t^2)}}{t + i \sqrt{z}} \, dt - \frac{1}{2 \pi i} \int_0^\infty \frac{f(-t^2)}{t - i \sqrt{z}} \, dt .
\]
Adding the corresponding sides of these identities, we find that
\[
\begin{aligned}
f(z) & = \frac{1}{2 \pi i} \int_0^\infty (\overline{f(-t^2)} - f(-t^2)) \biggl(\frac{1}{t - i \sqrt{z}} + \frac{1}{t + i \sqrt{z}} \biggr) dt \\
& = \frac{1}{2 \pi i} \int_0^\infty (-2 i \im f(-t^2)) \, \frac{2 t}{t^2 + z} \, dt = -\frac{1}{\pi} \int_0^\infty \frac{\im f(-s)}{s + z} \, ds .
\end{aligned}
\]
as desired. Similarly, subtracting the corresponding sides rather than adding them, we obtain
\[
\begin{aligned}
f(z) & = \frac{1}{2 \pi i} \int_0^\infty (\overline{f(-t^2)} + f(-t^2)) \biggl(\frac{1}{t - i \sqrt{z}} - \frac{1}{t + i \sqrt{z}} \biggr) dt \\
& = \frac{1}{2 \pi i} \int_0^\infty (2 \re f(-t^2)) \, \frac{2 i \sqrt{z}}{t^2 + z} \, dt = \frac{1}{\pi} \int_0^\infty \frac{\sqrt{z}}{\sqrt{s}} \, \frac{\re f(-s)}{s + z} \, ds ,
\end{aligned}
\]
as desired.
\end{proof}

%%%%%%%%%%%%%%%%%%%%%%%%%%%%%%%%%%%%%%%%%%%%%%%%%%%%%%%%%%%%%%%%%%%%%%%%%%%%%

\subsection{Auxiliary lemmas}

We need the following simple corollary of the residue theorem. 

\begin{lemma}[\cite{specHL}, p.~11, Lemma~2.5]
\label{lm:rotate}
Let $f$ be an analytic function in the sector $-\epsilon<\Arg(z)<b+\epsilon$ for some $\epsilon>0$ and $b>0$, except for a finite number of poles at points $z=z_j$ lying in the sector $0<\Arg(z)<b$. Assume also that for some $\delta>0$ we have $f(z)=O(|z|^{-1+\delta})$ as $|z|\rightarrow 0^+$ and $f(z)=O(|z|^{-1-\delta})$ as $|z|\rightarrow +\infty$, uniformly in the sector $0\leqslant \Arg(z)\leqslant b$. Then
\[
\int_0^\infty f(z)dz = e^{ib}\int_0^\infty f(e^{ib}z)dz+2\pi i\sum_j\Res(f(z_j)).
\]
\end{lemma}

We will need the following technical estimate.

\begin{lemma}
\label{lm:square:Kfg0}
Suppose that $\alpha \in (1, 2)$ and $h: (0,\infty)\rightarrow \RR$ satisfies $|h(r)|\leqslant c_1 \min \{ 1, r^{-1}\}$ and $|h^\prime(r)|\leqslant c_2\min \{1, r^{-2}\}$ if $r > 0$. Let
\[
K(s)=\pvint_0^\infty \frac{h(r)}{r^\alpha-s^\alpha} \, dr , \qquad s > 0.
\]
Then there is a constant $C$ (which depends only on $\alpha$, $c_1$ and $c_2$) such that
\begin{equation}
|K(s)|\leqslant \begin{cases} C s^{1 - \alpha} & \text{if $0 < s < 1$,} \\ C s^{-\alpha} \log(1 + s) & \text{if $s \geqslant 1$.} \end{cases}
\end{equation}
\end{lemma}
\begin{proof}
Fix $s > 0$. Since $|r^\alpha - s^\alpha| \ge r^\alpha - (\tfrac{1}{2} r)^\alpha \ge \tfrac{1}{2} r^\alpha + s^\alpha$ if $r \ge 2 s$, we have
\begin{equation}\label{eq:k:aux} |K(s)|\leqslant \biggl| \int_0^{2s} \frac{h(r) - h(s)}{r^\alpha-s^\alpha} \, dr \biggr| + \biggl| h(s) \pvint_0^{2s} \frac{1}{r^\alpha-s^\alpha} \, dr \biggr| + 2 \int_{2s}^\infty \frac{|h(r)|}{r^\alpha} \, dr . \end{equation}
The third integral in the right-hand side is easy to estimate: if $2 s \geqslant 1$, we have
\begin{equation}
\label{eq:Kfg0:1}
\int_{2s}^\infty \frac{|h(r)|}{r^\alpha} \, dr \le c_1 \int_{2s}^\infty \frac{1}{r^{\alpha + 1}} \, dr = \frac{c_1}{\alpha (2 s)^{\alpha}} \, ,
\end{equation}
while if $0 < 2 s < 1$,
\begin{equation}
\label{eq:Kfg0:2}
\int_{2s}^\infty \frac{|h(r)|}{r^\alpha} \, dr \le c_1 \int_{2s}^\infty \frac{1}{r^\alpha} \, dr = \frac{c_1}{(\alpha - 1) (2 s)^{\alpha - 1}} \, .
\end{equation}
The middle integral in~\eqref{eq:k:aux} also shows no difficulties:
\begin{equation}
\label{eq:Kfg0:3}
\biggl| h(s) \pvint_0^{2s} \frac{1}{r^\alpha-s^\alpha} \, dr \biggr| \le c_1 \min\{1, s^{-1}\} \times s^{1 - \alpha} \biggl| \pvint_0^2 \frac{1}{t^\alpha-1} \, dt \biggr| .
\end{equation}
The estimate of the first integral in~\eqref{eq:k:aux} requires more work. Since $|h'(t)| \le c_2 \min\{1, t^{-2}\} \le 2 c_2 (1 + t)^{-2}$, we have
\[
|h(r) - h(s)| = \biggl| \int_s^r h'(t) dt \biggr| \leqslant 2 c_2 \biggl| \frac{1}{1 + s} - \frac{1}{1 + r} \biggr| = \frac{2 c_2 |r - s|}{(1 + r) (1 + s)} \, .
\]
Therefore,
\[
\biggl| \int_0^{2s} \frac{h(r) - h(s)}{r^\alpha-s^\alpha} \, dr \biggr| \le 2 c_2 \int_0^{2s} \frac{r - s}{r^\alpha-s^\alpha} \, \frac{1}{(1 + r) (1 + s)} \, dr .
\]
Since $(t - 1) / (t^\alpha - 1) \le 1$ for all $t > 0$, we have
\[
\frac{r - s}{r^\alpha - s^\alpha} = \frac{1}{s^{\alpha - 1}} \, \frac{(r / s) - 1}{(r / s)^\alpha - 1} \le \frac{1}{s^{\alpha - 1}} \, .
\]
It follows that
\[
\biggl| \int_0^{2s} \frac{h(r) - h(s)}{r^\alpha-s^\alpha} \, dr \biggr| \le \frac{2 c_2}{s^{\alpha - 1}} \int_0^{2s} \frac{1}{(1 + r) (1 + s)} \, dr = \frac{2 c_2 \log(1 + 2 s)}{s^{\alpha - 1} (1 + s)} \, .
\]
Combining the above estimates, we conclude that
\[
K(s) \le \frac{C_1 \log(1 + 2 s)}{s^{\alpha - 1} (1 + s)} + C_2 s^{1 - \alpha} \min\{1, s^{-1}\} + C_3 \min\{s^{-\alpha}, s^{1 - \alpha}\}
\]
for some constants $C_1$, $C_2$, $C_3$. The desired result follows.
\end{proof}

The following identity is quite elementary.

\begin{lemma}[\cite{symzero1}, equation~(4.1)]
\label{lm:complex}
For every $a,b,c\in\CC$, $\Im c\neq 0$, we have
\[
\Im\left( \frac{ab}{c}\right)=\frac{\Im a\Im b}{\Im c}+\frac{\Im(a/c)\Im(b/c)}{\Im(1/c)}.
\]
\end{lemma}

We need one more technical result.

\begin{lemma}
\label{lem:aux:int}
Let $t > 0$ and
\[
 \Phi(z) = \int_0^\infty e^{-s^\alpha t} \, \frac{1 - e^{-s z}}{s} \, ds
\]
for $z \in \C$. Then $\Phi$ is an entire function, and for every $\epsilon > 0$ there is $C > 0$ such that
\[
 \biggl| \int_0^\infty e^{-s^\alpha t} \, \frac{e^{-s z_2} - e^{-s z_1}}{s} \, ds \biggr| = |\Phi(z_1) - \Phi(z_2)| \le C
\]
whenever $|z_1| = |z_2|$ and both $z_1$ and $z_2$ are in the sector $|\Arg z| \le \tfrac{\pi}{2} + \tfrac{\pi}{2 \alpha} - \epsilon$.
\end{lemma}

\begin{proof}
By Lemma~2.14 in~\cite{specHL} applied to the function $\exp(-s^\alpha t)$, the function
\[
 \Psi(z) = \int_0^\infty e^{-s^\alpha t} e^{-s z} ds
\]
is entire, and for every $\epsilon > 0$ there is $C > 0$ such that if $|\Arg z| \le \tfrac{\pi}{2} + \tfrac{\pi}{2 \alpha} - \epsilon$, then $|\Psi(z)| \le C \min\{1, |z|^{-1}\}$. Integrating $\Psi$ over $[0, z]$ and using Fubini's theorem, we find that
\[
 \int_{[0,z]} \Psi(w) dw = \int_0^\infty e^{-s^\alpha t} \int_{[0,z]} e^{-s w} dw ds = \int_0^\infty e^{-s^\alpha t} \, \frac{e^{-s z} - 1}{-s} \, ds = \Phi(z) ;
\]
the use of Fubini's theorem is justified by the estimate $|e^{-s^\alpha t} e^{-s w}| \le e^{-s^\alpha t} e^{s |z|}$. Therefore, $\Phi$ is indeed an entire function.

Suppose that $\eps > 0$, $r > 0$ and that $z_1 = r e^{i s_1}$ and $z_2 = r e^{i s_2}$ are in the sector $|\Arg z| \le \tfrac{\pi}{2} + \tfrac{\pi}{2 \alpha} - \epsilon$. Integrating over the arc $\Gamma$ of the circle $|z| = r$ with endpoints $z_1$ and $z_2$, we find that
\[
 |\Phi(z_2) - \Phi(z_1)| = \biggl| \int_\Gamma \Psi(w) dw \biggr| \le \int_{s_1}^{s_2} |\Psi(r e^{i s})| r ds \le 2 \pi r \times C \min\{1, r^{-1}\} , 
\]
which completes the proof.
\end{proof}

\section{Properties of the function $\Gpos$}
\label{sec:g}

Recall that the characteristic exponent of $X_t$ is given by $\psi(\xi) = (e^{-i \theta} \xi)^\alpha$ when $\xi > 0$ and $\psi(\xi) = (-e^{i \theta} \xi)^\alpha$ when $\xi < 0$, and these expressions extend analytically to $\re \xi > 0$ and $\re \xi < 0$, respectively.

\begin{proof}[Proof of Proposition~\ref{pr1}]
In terms of the characteristic exponent $\psi$, we need to prove that
\begin{equation}
\label{eq:fgpos}
\fourier \Gpos(\xi) = \sin \tfrac{\pi}{\alpha} \biggl(\frac{\alpha}{\psi(\xi)-1}-\frac{1}{e^{-i \theta} \xi - 1} + \frac{1}{e^{i \theta} \xi + 1} \biggr), \qquad \xi\in\RR .
\end{equation}
One way to prove the above identity is to simply evaluate the left-hand side using the definition~\eqref{eq:gdef} of $\Gpos$. We take a different approach: we apply the inverse Fourier transform to the right-hand side of~\eqref{eq:fgpos} and in this way we derive~\eqref{eq:gdef}.

We denote the right-hand side of~\eqref{eq:fgpos} by $\Phi(i \xi)$. Note that $|\Phi(i \xi)| \le C \min\{1, |\xi|^{-1}\}$. In particular, $\Phi(i \xi)$ is square integrable, and hence $\Phi(i \xi)$ indeed is the Fourier transform of a function $G \in \leb^2(\RR)$. Observe that $\Phi(-i \xi) = \overline{\Phi(i \xi)}$ for $\xi \in \RR$, and hence $G$ is real-valued.

By~\eqref{eq:fgpos}, $\Phi$ extends to an analytic function in the upper complex half-plane, continuous on the boundary, given by the formula
\[
\Phi(\xi) = \sin \tfrac{\pi}{\alpha} \biggl( \frac{\alpha}{(-i e^{-i \theta} \xi)^\alpha - 1} - \frac{e^{i \theta}}{-i \xi - e^{i \theta}} + \frac{e^{-i \theta}}{-i \xi + e^{-i \theta}} \biggr) , \qquad \im \xi \geqslant 0 .
\]
Indeed: the poles of $\alpha / (\psi(\xi) - 1)$ are cancelled by the other two terms. Furthermore, the estimate $|\Phi(\xi)| \le C \min\{1, |\xi|^{-1}\}$ holds in the upper complex half-plane, so that $\Phi$ is in the Hardy space $\mathscr{H}^2$ in the upper complex half-plane. Hence, by Theorem~\ref{thm:branges}, for $\xi > 0$ we have 
\begin{equation}\label{eq:pr1:aux1}
\fourier G(\xi) = \Phi(i \xi) = \frac{1}{2 \pi i} \int_{-\infty}^\infty \frac{\Phi(t)}{t - i \xi} \, dt , \qquad 0 = \frac{1}{2 \pi i} \int_{-\infty}^\infty \frac{\Phi(t)}{t + i \xi} \, dt .
\end{equation}
It follows that
\begin{equation}\label{eq:pr1:aux2}\begin{aligned}
\fourier G(\xi) & = \Phi(i \xi) + \overline{0} = \frac{1}{2 \pi i} \int_{-\infty}^\infty \frac{\Phi(t) - \overline{\Phi(t)}}{t - i \xi} \, dt \\
 & = \frac{1}{\pi} \int_{-\infty}^\infty \frac{\im \Phi(t)}{t - i \xi} \, dt = -\frac{1}{\pi} \int_0^\infty \frac{\im \Phi(-t)}{t + i \xi} \, dt + \frac{1}{\pi} \int_0^\infty \frac{\im \Phi(t)}{t - i \xi} \, dt
\end{aligned}\end{equation}
(here we added the sides of the first equality in~\eqref{eq:pr1:aux1} and complex conjugates of the corresponding sides of the other equality in~\eqref{eq:pr1:aux1}). By a straightforward calculation, for $t > 0$ we have
\[\begin{aligned}
\im \Phi(t) & = \sin \tfrac{\pi}{\alpha} \im \biggl( \frac{\alpha}{e^{-i \alpha (\pi/2 + \theta)} t^\alpha - 1} - \frac{1}{e^{-i (\pi/2 + \theta)} t - 1} - \frac{1}{e^{i (\pi/2 + \theta)} t - 1} \biggr) \\
& = \sin \tfrac{\pi}{\alpha} \times \frac{\alpha t^\alpha \sin(\alpha (\tfrac{\pi}{2} + \theta))}{t^{2 \alpha} - 2 t^\alpha \cos(\alpha (\tfrac{\pi}{2} + \theta)) + 1} \, .
\end{aligned}\]
Similarly, for $t > 0$,
\[
\im \Phi(-t) = -\sin \tfrac{\pi}{\alpha} \times \frac{\alpha t^\alpha \sin(\alpha (\tfrac{\pi}{2} - \theta))}{t^{2 \alpha} - 2 t^\alpha \cos(\alpha (\tfrac{\pi}{2} - \theta)) + 1} \, .
\]
In particular, $\im \Phi(t)$ is integrable over $t \in \RR$.

Observe that for $t > 0$, the functions $\xi \mapsto 1 / (t - i \xi)$ and $\xi \mapsto 1 / (t + i \xi)$ are Fourier transforms of $x \mapsto e^{t x} \chi_{(-\infty, 0)}(x)$ and $x \mapsto e^{-t x} \chi_{(0, \infty)}(x)$, respectively. By Fubini's theorem, \eqref{eq:gdef} and~\eqref{eq:pr1:aux2}, the Fourier transform of the function
\[
 \Gpos(x) = -\frac{1}{\pi} \int_0^\infty \im \Phi(-t) e^{-t x} \chi_{(0, \infty)}(x) dt + \frac{1}{\pi} \int_0^\infty \im \Phi(t) e^{t x} \chi_{(-\infty, 0)}(x) dt
\]
coincides with $\fourier G$ on $(0, \infty)$. Since both $\Gpos$ and $G$ are real-valued, we conclude that $\fourier \Gpos = \fourier G$ on $\RR$, and consequently $\Gpos$ and $G$ are equal almost everywhere.
\end{proof}

We will later see that
\[
 \Gpos(0^+) = \sin(\theta + \tfrac{\pi}{\alpha} - \tfrac{\pi}{2}) , \qquad \Gpos(0^-) = \sin(-\theta + \tfrac{\pi}{\alpha} - \tfrac{\pi}{2}) .
\]
In particular,
\[
 \Gpos(0^+) - \Gpos(0^-) = 2 \sin \tfrac{\pi}{\alpha} \sin \theta .
\]
We will need the following regularity result.

\begin{lemma}\label{lm:Hold}
The function $\Gpos$ is H\"older continuous with exponent $\alpha - 1$, save for a jump at $x = 0$. More precisely, the function
\[
g(x)=\Gpos(x) \chi_{\RR \setminus \{0\}}(x) - \Gpos(0^+)\chi_{(0,\infty)} - \Gpos(0^-)\chi_{(-\infty,0)}
\]
is H\"older continuous: there exists a constant $C$ such that
\begin{equation}
|g(x)-g(y)|\leqslant C|x-y|^{\alpha-1} , \qquad x, y \in \RR .
\end{equation}
\end{lemma}
\begin{proof}
We consider first an auxiliary function $f(x) = \Gpos(x) - c e^{-x} \chi_{(0, \infty)}(x)$, where
\[
c = 2 \sin \tfrac{\pi}{\alpha} \sin \theta = -i \sin \tfrac{\pi}{\alpha} (e^{i \theta} - e^{-i \theta}).
\]
Note that $f$ is continuous on $\RR \setminus \{0\}$, and
\[
\begin{aligned}
\fourier f(\xi) & = \sin \tfrac{\pi}{\alpha} \biggl(\frac{\alpha}{\psi(\xi)-1}-\frac{e^{i \theta}}{\xi - e^{i \theta}}+\frac{e^{-i \theta}}{\xi+e^{-i \theta}} \biggr) + \frac{c i}{\xi - i} \\
& = \sin \tfrac{\pi}{\alpha} \biggl(\frac{\alpha}{\psi(\xi)-1} - e^{i \theta} \biggl(\alpha \frac{1}{\xi-e^{i \theta}} - \frac{1}{\xi - i}\biggr) + e^{-i \theta} \biggl( \frac{1}{\xi+e^{-i \theta}} - \frac{1}{\xi - i} \biggr) \biggr) .
\end{aligned}
\]
It follows that $\fourier f(\xi)= O(|\xi|^{-\alpha})+O(|\xi|^{-2}) = O(|\xi|^{-\alpha})$ as $|\xi|\rightarrow \infty$, and therefore $|\fourier f(\xi)| \leqslant C_1 / (1 + |\xi|^\alpha)$ for $\xi \in \RR$. In particular, $\fourier f$ is integrable, and hence $f$, modified appropriately at $0$, is a continuous function. Furthermore, for $x,y\in\RR$,
\begin{equation}
\begin{split}
|f(x)-f(y)| & = \left|\int_{-\infty}^\infty (e^{i\xi x}-e^{i\xi y})\fourier f(\xi)d\xi \right|\leqslant \int_{-\infty}^\infty \min\{ 2, |x-y||\xi|\} \, \frac{C_1}{|\xi|^\alpha}\, d\xi \\
& = |x - y|^{\alpha - 1} \int_{-\infty}^\infty \min\{2, |t|\} \, \frac{C_1}{|t|^\alpha} \, dt = C_2 |x - y|^{\alpha - 1};
\end{split}
\end{equation}
we used a substitution $\xi = |x-y|^{-1}t$ in the penultimate step. Thus, $f$ is H\"older continuous with exponent $\alpha - 1$.

It remains to observe that $g - f$ (modified appropriately at zero) is bounded and Lipschitz continuous. Indeed, by definition, for some constants $c_1, c_2, c_3, c_4$ we have $g(x) - f(x) = (c_1+c_2 e^{-x}) \chi_{(0,\infty)}(x) + (c_3+c_4 e^x )\chi_{(-\infty,0)}(x)$ when $x \ne 0$. Since both $f$ and $g$ are continuous at zero, we necessarily have $c_1+c_2 = c_3+c_4$, and consequently $g - f$ is Lipschitz continuous.
\end{proof}

\section{Spectral expansion of transition operators $P_t^0$}

In this section we estabilish a generalised eigenfunction expansion for the transition operators $P_t^0$.

\subsection{Multiplication by test functions}

Recall that the Laplace transform of the transition density $p_t^0(x,y)$ with respect to $t$ is equal to the potential kernel $u^0_\lambda(x, y)$. Our first goal is to apply~\eqref{eq:uzero} to express the Laplace transform of $\int_{-\infty}^\infty \int_{-\infty}^\infty f(x) g(y) p^0_t(x, y) dx dy$ with respect to $t$ in terms of Fourier transforms of $f$ and $g$, where $f$ and $g$ are suitable test functions.

\begin{lemma}\label{lem:fgp}
If $f$ and $g$ are in both $\leb^1(\RR)$ and $\leb^2(\RR)$, then
\begin{equation}\label{eq:fgp}
\begin{aligned}
 & \int_0^\infty \int_{-\infty}^\infty \int_{-\infty}^\infty e^{-\lambda t} f(x) g(y) p^0_t(x, y) dx dy dt \\
 & \qquad = \frac{1}{2 \pi} \int_{-\infty}^\infty \frac{\fourier f(-\xi) \fourier g(\xi)}{\lambda + \psi(\xi)} \, d\xi \\
 & \qquad\qquad - \frac{1}{u_\lambda(0)} \biggl( \frac{1}{2 \pi} \int_{-\infty}^\infty \frac{\fourier f(-\xi)}{\lambda + \psi(\xi)} \, d\xi \biggr) \biggl( \frac{1}{2 \pi} \int_{-\infty}^\infty \frac{\fourier g(\eta)}{\lambda + \psi(\eta)} \, d\eta \biggr)
\end{aligned}
\end{equation}
for all $\lambda > 0$.
\end{lemma}

\begin{proof}
Recall that $p^0_t(x, y) \geqslant 0$ and $\int_0^\infty e^{-\lambda t} p^0_t(x, y) dt = u^0_\lambda(x, y) \le u_\lambda(y - x) \le u_\lambda(0)$. Thus, if $f$ and $g$ are integrable functions, then $e^{-\lambda t} f(x) g(y) p^0_t(x, y)$ is integrable with respect to $t > 0$ and $x, y \in \RR$. By Fubini's theorem and~\eqref{eq:uzero},
\[
\begin{aligned}
 & \int_0^\infty \int_{-\infty}^\infty \int_{-\infty}^\infty e^{-\lambda t} f(x) g(y) p^0_t(x, y) dx dy dt \\
 & \qquad = \int_{-\infty}^\infty \int_{-\infty}^\infty f(x) g(y) u^0_t(x, y) dx dy \\
 & \qquad = \int_{-\infty}^\infty \int_{-\infty}^\infty f(x) g(y) \biggl(u_\lambda(y-x)-\frac{u_\lambda(-x)u_\lambda(y)}{u_\lambda(0)}\biggr) dx dy .
\end{aligned}
\]
Suppose additionally that $f, g \in \leb^2(\RR)$. Since $\fourier u_\lambda(\xi) = 1 / (\lambda + \psi(-\xi))$ is in $\leb^2(\RR)$, we have $u_\lambda \in \leb^2(\RR)$ and, by Plancherel's theorem,
\[
\begin{aligned}
 \int_{-\infty}^\infty f(x) u_\lambda(y - x) dx & = \frac{1}{2 \pi} \int_{-\infty}^\infty e^{i \xi y} \fourier f(\xi) \fourier u_\lambda(\xi) d\xi , \\
 \int_{-\infty}^\infty g(y) u_\lambda(y) dy & = \frac{1}{2 \pi} \int_{-\infty}^\infty \fourier g(\eta) \fourier u_\lambda(-\eta) d\eta .
\end{aligned}
\]
It follows that
\[
\begin{aligned}
 & \int_0^\infty \int_{-\infty}^\infty \int_{-\infty}^\infty e^{-\lambda t} f(x) g(y) p^0_t(x, y) dx dy dt \\
 & \qquad = \frac{1}{2 \pi} \int_{-\infty}^\infty \biggl( \int_{-\infty}^\infty \fourier f(\xi) e^{i \xi y} g(y) \fourier u_\lambda(\xi) d\xi \biggr) dy \\
 & \qquad\qquad - \frac{1}{u_\lambda(0)} \biggl( \frac{1}{2 \pi} \int_{-\infty}^\infty \fourier f(\xi) \fourier u_\lambda(\xi) d\xi \biggr) \biggl( \frac{1}{2 \pi} \int_{-\infty}^\infty \fourier g(\eta) \fourier u_\lambda(-\eta) d\eta \biggr) .
\end{aligned}
\]
Since $\fourier f(\xi) \fourier u_\lambda(\xi)$ and $g(y)$ are integrable, once again applying Fubini's theorem, we eventually find that
\[
\begin{aligned}
 & \int_0^\infty \int_{-\infty}^\infty \int_{-\infty}^\infty e^{-\lambda t} f(x) g(y) p^0_t(x, y) dx dy dt \\
 & \qquad = \frac{1}{2 \pi} \int_{-\infty}^\infty \fourier f(\xi) \fourier g(-\xi) \fourier u_\lambda(\xi) d\xi \\
 & \qquad\qquad - \frac{1}{u_\lambda(0)} \biggl( \frac{1}{2 \pi} \int_{-\infty}^\infty \fourier f(\xi) \fourier u_\lambda(\xi) d\xi \biggr) \biggl( \frac{1}{2 \pi} \int_{-\infty}^\infty \fourier g(\eta) \fourier u_\lambda(-\eta) d\eta \biggr) .
\end{aligned}
\]
The desired result follows from $\fourier u_\lambda(\xi) = 1 / (\lambda + \psi(-\xi))$ after substituting $\xi$ for $-\xi$.
\end{proof}

In~\cite{symzero1}, the process $X_t$ is assumed to be symmetric, and so $\psi$ is real-valued. In this case inversion of the Laplace transform in $t$ in~\eqref{eq:fgp} is possible by extending analytically the right-hand side to $\lambda \in \C \setminus (-\infty, 0]$, and writing down a Stieltjes-like representation in terms of boundary values along $(-\infty, 0)$. In the non-symmetric case this approach is problematic: the right-hand side of~\eqref{eq:fgp} no longer automatically extends to an analytic function in $\C \setminus (-\infty, 0]$. A way around if found by considering more regular test functions, and deforming the contour of integration in $\xi$ and $\eta$ so that $\psi(\xi)$ and $\psi(\eta)$ are again real-valued.

\subsection{Contour deformation}\label{sec:contour}

Throughout this section, we fix $f,g\in\Ha$. Recall that $f$ and $g$ are real-valued, and their Laplace transforms are entire functions such that $\laplace f(\overline{\xi}) = \overline{\laplace f(\xi)}$, $\laplace g(\overline{\xi}) = \overline{\laplace g(\xi)}$. By Lemma~\ref{lm:7}, $|\laplace f(\xi)|$ and $|\laplace g(\xi)|$ are bounded by $C \min\{1, |\xi|^{-1}\}$ in every closed sector which does not contain neither $(0, \infty)$ nor $(-\infty, 0)$.

We will constantly use the following notation. For $r > 0$ we let
\[
\begin{aligned}
h_0(r)&=\cos\theta, \\ h_1(r)&=\Re(e^{i\theta}\laplace f(-ire^{i\theta})), \\ h_2(r)&=\Re(e^{i\theta}\laplace g(ire^{i\theta})) , \\ h_3(r)&=\Re(e^{i\theta}\laplace f(-ire^{i\theta})\laplace g(ire^{i\theta})) .
\end{aligned}
\]
For $j=0,1,2,3$ we define
\[
\phi_j(\lambda)=\frac{1}{\pi}\int_0^\infty \frac{h_j(r)}{r^\alpha+\lambda} \, dr, \qquad \lambda\in\CC\setminus (-\infty,0].
\]
Note that
\[
 \phi_0(\lambda) = \frac{\cos \theta}{\pi} \, \frac{1}{\lambda^{1 - 1/\alpha}} \, \int_0^\infty \frac{1}{t^\alpha + 1} \, dt = \frac{\cos \theta}{\alpha \sin \tfrac{\pi}{\alpha}} \, \frac{1}{\lambda^{1 - 1/\alpha}} \, \int_0^\infty \frac{1}{t^\alpha + 1} \, dt .
\]
Observe also that if $f(x) = g(-x)$, then $h_2(r) = h_1(r)$. Finally, we set
\[
\phi_4(\lambda)=\phi_3(\lambda)-\frac{\phi_1(\lambda)\phi_2(\lambda)}{\phi_0(\lambda)} \, ,\qquad \lambda\in\CC\setminus (-\infty,0],\]
and for later needs we extend the above definitions to $(-\infty, 0)$ by the formula
\[
\phi_j(-\lambda)\mdef\lim_{\varepsilon\rightarrow 0^+}\phi_j(-\lambda+i\varepsilon), \qquad \lambda>0.
\]
The following result is a variant of Lemma~\ref{lem:fgp} after appropriate contour deformation.

\begin{lemma}\label{lem:fgp2}
With the above assumptions and notation, we have
\begin{equation}\label{eq:fgp2}
\int_0^\infty \int_{-\infty}^\infty \int_{-\infty}^\infty e^{-\lambda t} f(x) g(y) p^0_t(x, y) dx dy dt = \phi_4(\lambda), \qquad \lambda > 0 .
\end{equation}
\label{lm:deformacja}
\end{lemma}
\begin{proof}
Fix $\lambda > 0$ and denote the left-hand side of~\eqref{eq:fgp2} by $I$. By Lemma~\ref{lem:fgp},
\begin{equation}\label{eq:fgp2:aux}
\begin{aligned}
 I & = \frac{1}{2 \pi} \int_{-\infty}^\infty \frac{\laplace f(-i \xi) \laplace g(i \xi)}{\lambda + \psi(\xi)} \, d\xi \\
 & \qquad - \frac{1}{u_\lambda(0)} \biggl( \frac{1}{2 \pi} \int_{-\infty}^\infty \frac{\laplace f(-i \xi)}{\lambda + \psi(\xi)} \, d\xi \biggr) \biggl( \frac{1}{2 \pi} \int_{-\infty}^\infty \frac{\laplace g(i \eta)}{\lambda + \psi(\eta)} \, d\eta \biggr) .
\end{aligned}
\end{equation}
Now we deform the contour of integration $\RR$ to $(-e^{-i \theta} \infty, 0) \cup (0, e^{i \theta} \infty)$ in each of the three integrals in the right-hand side. Recall that $\psi(e^{i \theta} r) = \psi(-e^{-i \theta} r) = r^\alpha$. We use Lemma~\ref{lm:rotate}: since 
\[\begin{gathered}
|\laplace f(-i\xi)| \le C \min(1, |\xi|^{-1}) , \qquad |\laplace g(i \xi)| \le C \min(1, |\xi|^{-1}) , \\ \frac{1}{|\lambda + \psi(\xi)|} \le C \min\{1, |\xi|^{-\alpha}\}
\end{gathered}\]
in the sector $\{\xi \in \C : |\arg(\xi)| \le |\theta|\}$ (see Lemma~\ref{lm:7}), we have
\[
\begin{aligned}
\frac{1}{2 \pi} \int_{(0, \infty)} \frac{\laplace f(-i \xi) \laplace g(i \xi)}{\lambda + \psi(\xi)} \, d\xi & = 
\frac{1}{2 \pi} \int_{(0, e^{i \theta} \infty)} \frac{\laplace f(-i \xi) \laplace g(i \xi)}{\lambda + \psi(\xi)} \, d\xi \\
& = \frac{1}{2 \pi} \int_0^\infty \frac{e^{i \theta} \laplace f(-i e^{i \theta} r) \laplace g(i e^{i \theta} r)}{\lambda + r^\alpha} \, dr .
\end{aligned}
\]
In a similar way,
\[
\begin{aligned}
\frac{1}{2 \pi} \int_{(-\infty, 0)} \frac{\laplace f(-i \xi) \laplace g(i \xi)}{\lambda + \psi(\xi)} \, d\xi & = 
\frac{1}{2 \pi} \int_{(-e^{-i \theta} \infty, 0)} \frac{\laplace f(-i \xi) \laplace g(i \xi)}{\lambda + \psi(\xi)} \, d\xi \\
& = \frac{1}{2 \pi} \int_0^\infty \frac{e^{-i \theta} \laplace f(i e^{-i \theta} r) \laplace g(-i e^{-i \theta} r)}{\lambda + r^\alpha} \, dr .
\end{aligned}
\]
It follows that
\[
\begin{aligned}
& \frac{1}{2 \pi} \int_{-\infty}^\infty \frac{\laplace f(-i \xi) \laplace g(i \xi)}{\lambda + \psi(\xi)} \, d\xi \\
& \qquad = \frac{1}{2 \pi} \int_0^\infty \frac{e^{-i \theta} \laplace f(i e^{-i \theta} r) \laplace g(-i e^{-i \theta} r) + e^{-i \theta} \laplace f(i e^{-i \theta} r) \laplace g(-i e^{-i \theta} r)}{\lambda + r^\alpha} \, dr \\
& \qquad = \frac{1}{\pi} \int_0^\infty \frac{\re(e^{-i \theta} \laplace f(i e^{-i \theta} r) \laplace g(-i e^{-i \theta} r))}{\lambda + r^\alpha} \, dr = \phi_3(\lambda) .
\end{aligned}
\]
The same argument applies to the other two integrals in the right-hand side of~\eqref{eq:fgp2:aux}, which are found to be equal to $\phi_1(\lambda)$ and $\phi_2(\lambda)$, respectively. Finally, again by the same argument,
\begin{equation}
\begin{split}
u_\lambda(0) & = \frac{1}{2\pi}\int_{-\infty}^\infty \frac{1}{\lambda + \psi(-\eta)} \, d\eta =\frac{1}{2\pi}\int_0^\infty \frac{e^{i\theta} + e^{-i \theta}}{\lambda+r^\alpha} \, dr = \phi_0(\lambda).
\end{split}
\end{equation}
We conclude that
\begin{equation}
I = \phi_3(\lambda)-\frac{\phi_1(\lambda) \phi_2(\lambda)}{\phi_0(\lambda)} = \phi_4(\lambda), \qquad \lambda > 0 ,
\label{precauchy}
\end{equation}
and the proof is complete.
\end{proof}

\subsection{Application of the Cauchy's integral formula}\label{sec:cauchy}
We consider $f, g \in \Ha$, and we continue to use the notation introduced in the previous section. We now use Cauchy's integral formula given in Theorem~\ref{thm:branges} for the function $\sqrt{\xi} \phi_4(\xi)$. The following set of lemmas justify the application of this result. First we check that $\phi_4\in \mathscr{N}^+$.

\begin{lemma}
\label{lm:bounded.type}
The function $\lambda \phi_4(-\lambda^2)$ is in the Nevanlinna class $\mathscr{N}^+$.
\end{lemma}

\begin{proof}
Recall that for $j = 0, 1, 2, 3$ and $\lambda \in \C \setminus (-\infty, 0]$, we have
\[
\phi_j(\lambda) = \frac{1}{\pi} \int_0^\infty \frac{h_j(r)}{r^\alpha+\lambda} \, dr = \frac{1}{\alpha \pi} \int_0^\infty \frac{s^{1/\alpha - 1} h_j(s^{1/\alpha})}{s + \lambda} \, ds .
\]
Therefore, if $\im \lambda > 0$, then
\begin{equation}
\label{eq:phi:stieltjes}
\begin{aligned}
\lambda \phi_j(-\lambda^2) & = \frac{1}{\alpha \pi} \int_0^\infty \frac{\lambda s^{1/\alpha - 1} h_j(s^{1/\alpha})}{s - \lambda^2} \, ds \\
& = \frac{1}{2 \alpha \pi} \int_0^\infty \frac{s^{1/\alpha - 1} h_j(s^{1/\alpha})}{s - \lambda} \, ds - \frac{1}{2 \alpha \pi} \int_0^\infty \frac{s^{1/\alpha - 1} h_j(s^{1/\alpha})}{s + \lambda} \, ds .
\end{aligned}
\end{equation}
Writing $h_j(r) = \max\{h_j(r), 0\} - \max\{-h_j(r), 0\}$, we see that $\lambda \phi_j(-\lambda^2 \lambda) = \phi_{j,1}(-\lambda) - \phi_{j,2}(-\lambda) - \phi_{j,3}(\lambda) + \phi_{j,4}(\lambda)$ for appropriate Stieltjes functions $\phi_{j,1}, \phi_{j,2}, \phi_{j,3}, \phi_{j,4}$. By Lemmas~\ref{lm:BTcircle} and~\ref{lm:sum.is.N}, $\lambda \phi_j(-\lambda^2)$ is in the Nevanlinna class $\mathscr{N}^+$.

Similarly, we show that $\lambda \phi_0(-\lambda^2)$ is an outer function. We have $h_0(s^{1/\alpha}) \geqslant 0$ in~\eqref{eq:phi:stieltjes}, and hence $\im(\lambda \phi_0(-\lambda^2)) \geqslant 0$ whenever $\im \lambda > 0$. By Lemma~\ref{lm:BTcircle}, $-i\lambda \phi_0(-\lambda^2)$ is an outer function, and hence also $\lambda \phi_0(-\lambda^2)$ is an outer function.

Lemma~\ref{lm:sum.is.N} implies now that the function $\phi_4 = \phi_3 - \phi_1\phi_2 / \phi_0$ is in the Nevanlinna class $\mathscr{N}^+$.
\end{proof}

Next, we study the boundary values of $\phi_4$. This will involve the following pricipal value integrals:
\[
K_j(s)=\frac{1}{\pi}\pvint_0^\infty \frac{h_j(r)dr}{r^\alpha-s^\alpha}, \qquad s > 0 ,
\]
where $j = 0, 1, 2, 3$. Note that since $h_0(r) = \cos \theta$, we have
\[
K_0(s)=-\frac{\cot\frac{\pi}{\alpha} \cos\theta}{\alpha s^{\alpha - 1}} \, , \qquad s > 0 ;
\]
see, for example, \cite{symzero1}, Example 5.1.
For $j = 0, 1, 2, 3$, define also
\[
L_j(s)=\frac{h_j(s)}{\alpha s^{\alpha-1}}\, , \qquad s > 0.
\]
In particular, 
\[
L_0(s) = \frac{\cos \theta}{\alpha s^{\alpha - 1}} \, , \qquad s > 0.
\]

\begin{lemma}\label{lem:imaginaryPart}
With the above notation, $\phi_4$ extends to a continuous function in the closed upper complex half-plane, except possibly at $0$. If this extension is denoted by the same symbol $\phi_4$, then we have
\begin{equation}
\label{eq:imaginaryPart}
\begin{split}
\Im\phi_4(-s^\alpha) & = -L_3(s)+\frac{ L_1(s)L_2(s)}{L_0(s)}\\
 & - \biggl(\Im\frac{1}{K_0(s)-iL_0(s)}\biggr)^{-1} \Im\frac{K_1(s)- iL_1(s)}{K_0(s)-iL_0(s)} \, \Im \frac{ K_2(s)-iL_2(s)}{K_0(s)-iL_0(s)} .
\end{split}
\end{equation}
\end{lemma}

\begin{proof}
Note that $h_j$, $j = 0, 1, 2, 3$, are continuously differentiable functions on $(0, \infty)$. Thus, $\phi_j$, $j = 0, 1, 2, 3$, extend to continuous functions in the upper complex half-plane, except possibly at $0$; we denote these extensions again by $\phi_j$. Furthermore, by Sokhozki's formula, for $j=0,1,2,3$ and $s > 0$ we have
\[
\begin{aligned}
\phi_j(-s^\alpha) & = \lim_{\varepsilon \rightarrow 0^+}\phi_j(-s^\alpha+i\varepsilon)= \lim_{\varepsilon \rightarrow 0^+}\frac{1}{\pi}\int_0^\infty  \frac{h_j(r)}{r^\alpha-s^\alpha + i\varepsilon} \, dr \\
 &= \lim_{\varepsilon \rightarrow 0^+} \frac{1}{\pi}\int_0^\infty\frac{r^\alpha-s^\alpha}{(r^\alpha-s^\alpha)^2+\varepsilon^2} \, h_j(r)dr- \lim_{\varepsilon \rightarrow 0^+}\frac{1}{\pi}\int_0^\infty\frac{i\varepsilon}{(r^\alpha-s^\alpha)^2+\varepsilon^2} \, h_j(r)dr\\
 & = \frac{1}{\pi} \pvint_0^\infty \frac{h_j(r)dr}{r^\alpha-\lambda(s)}-\frac{i}{\alpha s^{\alpha-1}}\frac{\pi h_j(s)}{\pi}=K_j(s)-iL_j(s);
\end{aligned}
\]
see~\cite{vladimirov}, Section 1.8, or a similar calculation in the proof of Lemma 4.3 in~\cite{symzero1}. Hence, we can express $\phi_4(-s^\alpha)$ as
\begin{equation}
\label{exp:phinegative}
\phi_4(-s^\alpha)= K_3(s)-i L_3(s)-\frac{(K_1(s)- iL_1(s)) (K_2(s)-iL_2(s))}{K_0(s)-iL_0(s)}.
\end{equation}
Now, since $K_j$ and $L_j$ are real-valued for $j = 0, 1, 2, 3$, the desired result follows by Lemma~\ref{lm:complex}.
\end{proof}

The following estimates will be used to prove square-integrability of $s \phi_4(-s^2)$ for the application of Theorem~\ref{thm:branges}.
\begin{lemma}
\label{lm:square:Kfg}
With the above notation and assumptions, there is a constant $C$ such that for $j = 1, 2, 3$ we have
\begin{equation}
|K_j(s)|\leqslant \begin{cases} C s^{1 - \alpha} & \text{if $0 < s < 1$,} \\ C s^{-\alpha} \log(1 + s) & \text{if $s \geqslant 1$} \end{cases}
\end{equation}
\end{lemma}
\begin{proof}
By Lemmas~\ref{lm:7} and~\ref{lm:7p}, the functions $h_1$ and $h_2$ satisfy the assumptions of Lemma~\ref{lm:square:Kfg0}, and consequently
\[
 |K_j(s)| \le \begin{cases} C s^{1 - \alpha} & \text{if $0 < s < 1$,} \\ C s^{-\alpha} \log(1 + s) & \text{if $s \geqslant 1$} \end{cases} 
\]
for $j = 1, 2$. Similarly, again by Lemmas~\ref{lm:7} and~\ref{lm:7p}, we have $|h_3(r)| \le C \min\{1, r^{-2}\}$ and $|h_3'(r)| \le C \min\{1, r^{-3}\}$, so Lemma~\ref{lm:square:Kfg0} applies also $h_3$, leading to the desired bound for $K_3$.
\end{proof}
%%%%%%%%%%%%%%%%%%%%%%%%%%%%%%%%%%%%%%%%%%%%%%%%%%%%%%%%%%%%%%%%%%%%%%%%%%%%%%%%%%%%%%%%%%
%%%%%%%%%%%%%%%%%%%%%%%%%%%%%%%%%%%%%%%%%%%%%%%%%%%%%%%%%%%%%%%%%%%%%%%%%%%%%%%%%%%%%%%%%%

\begin{lemma}
\label{lem:phi4}
With the above notation and assumptions, there is a constant $C$ such that
\begin{equation}
|\phi_4(-\lambda)| \leqslant \begin{cases} C \lambda^{1/\alpha - 1} & \text{if $0 < \lambda < 1$,} \\ C \lambda^{-1} (\log(1 + \lambda))^2 & \text{if $\lambda \geqslant 1$} \end{cases}
\end{equation}
\end{lemma}

\begin{proof}
For $s > 0$, we have
\[
\begin{aligned}
 |\phi_4(-s^\alpha)| & = \biggl|\phi_3(-s^\alpha) - \frac{\phi_1(-s^\alpha) \phi_2(-s^\alpha)}{\phi_0(-s^\alpha)}\biggr| \\
 & \leqslant |K_3(s)| + |L_3(s)| + \frac{(|K_1(s)| + |L_1(s)|) (|K_2(s)| + |L_2(s)|)}{|L_0(s)|} \, . 
\end{aligned}
\]
Recall that $L_0(s) = \cos \theta / (\alpha s^{\alpha - 1})$ for $s > 0$. By definition and Lemma~\ref{lm:7}, there is a constant $C_1$ such that if $s > 0$, then $|L_j(s)| \le C_1 s^{1 - \alpha} \min\{1, s^{-1}\}$ for $j = 1, 2$, and $|L_3(s)| \le C_1 s^{1 - \alpha} \min\{1, s^{-2}\}$. Similar estimates for $K_j(s)$, $j = 1, 2, 3$, are given in Lemma~\ref{lm:square:Kfg}. It follows that for some constant $C_2$ we have
\[
 |\phi_4(-s^\alpha)| \le C_2 s^{1 - \alpha} , \qquad 0 < s < 1 ,
\]
and
\[
 |\phi_4(-s^\alpha)| \le C_2 s^{-\alpha} (\log(1 + s))^2 , \qquad s \geqslant 1 ,
\]
as desired.
\end{proof}

We are now ready to apply Theorem~\ref{thm:branges}.

\begin{theorem} Let $f,g\in\Ha$ and $t > 0$. Then, with the above notation,
\label{thm:imagine}
\begin{equation}
\label{eq:imagine}
\int_{-\infty}^\infty \int_{-\infty}^\infty f(x) g(y) p^0_t(x, y) dx dy = -\frac{1}{\pi} \int_0^\infty e^{-r t} \im \phi_4(-r) dr ,
\end{equation}
where $\im \phi_4$ is given by Lemma~\ref{lem:imaginaryPart}.
\end{theorem}
\begin{proof}
Recall that by Lemma~\ref{lm:deformacja},
\[
\int_0^\infty \int_{-\infty}^\infty \int_{-\infty}^\infty e^{-\lambda t} f(x) g(y) p^0_t(x, y) dx dy dt = \phi_4(\lambda), \qquad \lambda > 0 .
\]
We verify that the function $\sqrt{\lambda} \phi_4(\lambda)$ satisfies the assumptions of Corollary~\ref{cor:branges}. By Lemma~\ref{lm:bounded.type}, the function $\lambda \phi_4(\lambda^2)$ is in the Nevanlinna class $\mathscr{N}^+$. By Lemma~\ref{lem:imaginaryPart}, this function extends continuously to the closed upper complex half-plane, except possibly at $\lambda = 0$, and, by Lemma~\ref{lem:phi4}, this extension satisfies
\[
|\sqrt{\lambda} \phi_4(-\lambda)| \leqslant \begin{cases} C \lambda^{1/\alpha - 1/2} & \text{if $0 < \lambda < 1$,} \\ C \lambda^{-1/2} (\log(1 + \lambda))^2 & \text{if $\lambda \geqslant 1$.} \end{cases}
\]
In particular,
\[
\int_0^\infty |\lambda \phi_4(-\lambda^2)|^2 d\lambda \le C^2 \int_0^1 \lambda^{2/\alpha - 1} d\lambda + C^2 \int_1^\infty \lambda^{-2} (\log(1 + \lambda^2))^4 d\lambda < \infty .
\]
Therefore, the assumptions of Corollary~\ref{cor:branges} are satisfied.

We conclude that for $\lambda \in \C \setminus (-\infty, 0]$,
\[
\sqrt{\lambda} \phi_4(\lambda) = \frac{1}{\pi} \int_0^\infty \frac{\sqrt{\lambda}}{\sqrt{r}} \, \frac{\re (i \sqrt{r} \phi_4(-r))}{r + \lambda} \, dr = -\frac{\sqrt{\lambda}}{\pi} \int_0^\infty \frac{\im \phi_4(-r)}{r + \lambda} \, dr ,
\]
which implies that for $\lambda > 0$,
\[
\int_0^\infty \int_{-\infty}^\infty \int_{-\infty}^\infty e^{-\lambda t} f(x) g(y) p^0_t(x, y) dx dy dt = -\frac{1}{\pi} \int_0^\infty \int_0^\infty \im \phi_4(-r) e^{-r t} e^{-\lambda r} dt dr .
\]
The desired result for almost every $t > 0$ follows by Fubini's theorem and uniqueness of Laplace transforms. Extension to all $t > 0$ is a consequence of continuity. Indeed, the right-hand side of~\eqref{eq:imagine} is clearly continuous in $t > 0$. Continuity of the left-hand side results from integrability of $f(x) g(y)$ with respect to $x, y \in \RR$, and continuity of $t \mapsto p^0_t(\cdot, \cdot)$ on $(0, \infty)$ with respect to the topology of uniform convergence. %TODO: więcej szczegółów?
\end{proof}
%%%%%%%%%%%%%%%%%%%%%%%%%%%%%%%%%%%%%%%%%%%%%%%%%%%%%%%%%%%%%%%%%%
%%%%%%%%%%%%%%%%%%%%%%%%%%%%%%%%%%%%%%%%%%%%%%%%%%%%%%%%%%%%%%%%%%
\subsection{Generalised eigenfunction expansion of $P^0_t$}
Our goal in this section is to express $\im \phi_4(-\lambda)$ (see Lemma~\ref{lem:imaginaryPart}) in terms of $f$ and $g$ rather than the Laplace transforms of $f$ and $g$. This result, combined with Theorem~\ref{thm:imagine}, will prove Theorem~\ref{thm:1}. We use the notation introduced in Sections~\ref{sec:contour} and~\ref{sec:cauchy}.

The expression for $\im \phi_4(-\lambda)$ in Lemma~\ref{lem:imaginaryPart}:
\[
\begin{aligned}
\Im\phi_4(-s^\alpha) & = -L_3(s)+\frac{ L_1(s)L_2(s)}{L_0(s)}\\
 & - \biggl(\Im\frac{1}{K_0(s)-iL_0(s)}\biggr)^{-1} \Im\frac{K_1(s)- iL_1(s)}{K_0(s)-iL_0(s)} \, \Im \frac{ K_2(s)-iL_2(s)}{K_0(s)-iL_0(s)} ,
\end{aligned}
\]
has two parts. The former one is expanded in Lemma~\ref{lm:abc}; the latter one is more involved and it is studied in Lemma~\ref{lm:oddeformuj2}, after a number of auxiliary results.

\begin{lemma} 
\label{lm:abc}
With the above assumptions and notation,
\[
\begin{aligned}
& L_3(s)-\frac{L_1(s)L_2(s)}{L_0(s)}
\\ & \qquad = \frac{1}{\alpha s^{\alpha-1}\cos\theta} \biggl( \int_{-\infty}^\infty e^{-sx\sin\theta} \sin(sx\cos\theta)f(x)dx\biggr) \biggl( \int_{-\infty}^\infty e^{sy\sin\theta}\sin(sy\cos\theta)g(y)dy\biggr).
\end{aligned}
\]
\end{lemma}

\begin{proof}
Recall that for $j = 0, 1, 2, 3$, $L_j(s) = h_j(s) / (\alpha s^{\alpha - 1})$, where
\[
\begin{aligned} h_0(s) & = \cos \theta, \qquad & h_1(s) & = \re (e^{i \theta} \laplace f(-i s e^{i \theta})) , \\ h_2(s) & = \re (e^{i \theta} \laplace g(i s e^{i \theta})) , \qquad & h_3(s) & = \re(e^{i \theta} \laplace f(-i s e^{i \theta}) \laplace g(i s e^{i \theta})). \end{aligned}
\]
For a fixed $s > 0$, denote $\Re\laplace f(-i s e^{i\theta}) = a_1$, $\Im\laplace f(-i s e^{i\theta}) = b_1$, $\Re\laplace g(i s e^{i\theta}) = a_2$, $\Im\laplace g(i s e^{i\theta}) = b_2$. We have
\[
\begin{aligned}
 L_1(s) & =\frac{\re(e^{i\theta}\laplace f(-ise^{i\theta}))}{\alpha s^{\alpha-1}} = \frac{a_1 \cos\theta - b_1 \sin\theta}{\alpha s^{\alpha-1}} \, , \\
 L_2(s) & =\frac{\re(e^{i\theta}\laplace g(ise^{i\theta}))}{\alpha s^{\alpha-1}}= \frac{a_2 \cos\theta - b_2 \sin\theta}{\alpha s^{\alpha-1}} \, , \\
 L_3(s) & =\frac{\re(e^{i\theta}\laplace f(-ise^{i\theta})\laplace g(ise^{i\theta}))}{\alpha s^{\alpha-1}} = \frac{(a_1 a_2 - b_1 b_2) \cos\theta - (a_1 b_2 + a_2 b_1) \sin \theta}{\alpha s^{\alpha-1}} \, .
\end{aligned}
\]
It follows that
\[\begin{aligned}
L_3(r)-\frac{L_1(r)L_2( r)}{L_0(r)} & = \frac{(a_1 a_2 - b_1 b_2) \cos\theta - (a_1 b_2 + a_2 b_1) \sin \theta}{\alpha s^{\alpha-1}} \\
& \qquad - \frac{(a_1 \cos\theta - b_1 \sin\theta) (a_2 \cos\theta - b_2 \sin\theta)}{\alpha s^{\alpha-1} \cos \theta} \\
& = -\frac{b_1 b_2 \cos\theta}{\alpha s^{\alpha-1}} - \frac{b_1 b_2 \sin^2 \theta)}{\alpha s^{\alpha-1} \cos \theta} = -\frac{b_1 b_2}{\alpha s^{\alpha-1} \cos \theta} \, .
\end{aligned}\]
Since
\[\begin{aligned}
b_1 & = \Im \laplace f(-ise^{i\theta}) = \int_{-\infty}^\infty e^{-sx\sin\theta} \sin(sx\cos\theta)f(x)dx, \\
b_2 & = \Im \laplace g(ise^{i\theta}) = -\int_{-\infty}^\infty e^{sy\sin\theta} \sin(sy\cos\theta)g(y)dy,
\end{aligned}\]
the proof is complete.
\end{proof}
%%%%%%%%%%%%%%%%%%%%%%%%%%%%%%%%%%%%%%%%%%%%%%%%%%%%%%%%%%%%%%%%%%%%

Recall that every $f\in\Ha$ can be written as $f=f_++f_-$, where $f_+ = f \chi_{(0, \infty)} \in\Ha_+$ and $f_- = f \chi_{(-\infty, 0)} \in\Ha_-$. For the next result, we need the following variant of Sokhozki's formula.
\begin{lemma}
\label{lm:pv}
If $s > 0$, $\zeta = s e^{i \theta}$ or $\zeta = -s e^{-i \theta}$, $f\in\Ha$ and $f = f_+ + f_-$ as above, then
\begin{equation}\label{eq:pv}
\frac{1}{2 \pi i} \pvint_{(-e^{-i \theta} \infty, 0) \cup (0, e^{i \theta} \infty)} \frac{\laplace f(-i \xi)}{\xi - \zeta} \, d\xi = \frac{1}{2} \, \laplace f_+(-i \zeta) - \frac{1}{2} \, \laplace f_-(-i \zeta) .
\end{equation}
\end{lemma}
\begin{proof}
Let $\Gamma_R$ denote the boundary of the circular sector 
\[
D_R = \{\xi \in \C : |\xi| \le R , \, |\arg (-i \xi)| \le \tfrac{\pi}{2} - \theta\},
\]
oriented in a counter-clockwise manner. Since $|\laplace f_+(-i \xi)| \le C \min\{1, |\xi|^{-1}\}$ in $D_R$ by Lemma~\ref{lm:7}, we have
\begin{equation}\label{eq:pv:aux}
\frac{1}{2 \pi i} \pvint_{(-e^{-i \theta} \infty, 0) \cup (0, e^{i \theta} \infty)} \frac{\laplace f_+(-i \xi)}{\xi - \zeta} \, d\xi = \lim_{R \to \infty} \frac{1}{2 \pi i} \pvint_{\Gamma_R} \frac{\laplace f_+(-i \xi)}{\xi - \zeta} \, d\xi .
\end{equation}
If $R > s$, then $\laplace f_+(-i \xi)$ is analytic in the neighbourhood of $D_R$ and $\zeta \in \Gamma_R$. Thus, by the usual Sokhozki's formula, the expression under the limit in the above equality is equal to $\tfrac{1}{2} \laplace f_+(-i \zeta)$.

A similar argument applies to $f_-$ rather than $f_+$, but here we need to consider the boundary $\Gamma_R$ of the circular sector 
\[
D_R = \{\xi \in \C : |\xi| \le R , \, |\arg (i \xi)| \le \tfrac{\pi}{2} + \theta\},
\]
oriented in a clockwise manner. Note that $|\laplace f_-(-i \xi)| \le C \min\{1, |\xi|^{-1}\}$ in $D_R$, and so we have a complete analogue of~\eqref{eq:pv:aux} for $f_-$. However, the expression under the limit in the right-hand side is now equal to $-\tfrac{1}{2} \laplace f_-(-i \zeta)$ (with a minus sign) due to clockwise orientation of $\Gamma_R$. The assertion of the lemma follows by combining the above results for $f_+$ and $f_-$.
\end{proof}

Observe that for $f \in \Ha$ and $r > 0$,
\[
 e^{i \theta} \laplace f(-ise^{i\theta}) = \int_{-\infty}^\infty e^{-s x \sin\theta} (\cos(s x \cos\theta + \theta) + i \sin(s x \cos \theta + \theta)) f(x)dx .
\]

%%%%%%%%%%%%%%%%%%%%%%%%%%%%%%%%%%%%%  ODDEFORMUJ %%%%%%%%%%%%%%%%%%%%%%%%%%%%%%%
\begin{lemma}
\label{lm:oddeformuj}
With the above assumptions and notation, for $s > 0$ we have
\begin{subequations}
\begin{equation}
K_1(s)= \frac{1}{\alpha s^{\alpha - 1}} \int_{-\infty}^\infty \left(\frac{\Gpos(sx)}{\sin \tfrac{\pi}{\alpha}} - e^{-sx\sin\theta}\sin(sx\cos\theta+\theta)\sign x \right)f(x)dx,
\end{equation}
\begin{equation}
K_2(s)= \frac{1}{\alpha s^{\alpha - 1}} \int_{-\infty}^\infty \left(\frac{\Gpos(-sx)}{\sin \tfrac{\pi}{\alpha}} - e^{sx\sin\theta}\sin(sx\cos\theta-\theta)\sign x \right)g(x)dx.
\end{equation}
\end{subequations}
\end{lemma}
\begin{proof}
Fix $s > 0$. By definition,
\[
\begin{aligned}
 K_1(s) & = \frac{1}{\pi}\pvint_0^\infty \frac{\re(e^{i\theta} \laplace f(-ire^{i\theta}))}{r^{\alpha}-s^\alpha} \, dr \\
 & = \frac{1}{2\pi} \pvint_0^\infty \frac{e^{i\theta} \laplace f(-ire^{i\theta})}{r^{\alpha}-s^\alpha} \, dr + \frac{1}{2 \pi} \pvint_0^\infty \frac{e^{-i\theta} \laplace f(ire^{-i\theta})}{r^\alpha-s^\alpha} \, dr \\
 & = \frac{1}{2 \pi} \pvint_{(-e^{-i \theta} \infty, 0) \cup (0, e^{i \theta} \infty)} \frac{\laplace f(-i \xi)}{\psi(\xi) - s^\alpha} \, d\xi .
\end{aligned}
\]
The function $\psi(\xi) - s^\alpha$ is meromorphic in $\C \setminus i \RR$, with two simple poles at $\xi = s e^{i \theta}$ and $\xi = -s e^{-i \theta}$. The corresponding residues are $1 / \psi'(s e^{i \theta}) = e^{i \theta} / (\alpha s^{\alpha - 1})$ and $1 / \psi'(-s e^{-i \theta}) = -e^{-i \theta} / (\alpha s^{\alpha - 1})$, respectively. Removing this poles leads to the identity
\begin{equation}\label{eq:k1:aux}
\begin{aligned}
 K_1(s) & = \frac{1}{2 \pi} \int_{(-e^{-i \theta} \infty, 0) \cup (0, e^{i \theta} \infty)} \Phi(\xi) \laplace f(-i \xi) d\xi \\
 & \hspace*{-1em} + \frac{1}{2 \pi} \pvint_{(-e^{-i \theta} \infty, 0) \cup (0, e^{i \theta} \infty)} \biggl(\frac{e^{i \theta}}{\alpha s^{\alpha - 1}} \, \frac{1}{\xi - s e^{i \theta}} - \frac{e^{-i \theta}}{\alpha s^{\alpha - 1}} \, \frac{1}{\xi + s e^{-i \theta}}\biggr) \laplace f(-i \xi) d\xi ,
\end{aligned}\end{equation}
where
\[
\begin{aligned}
 \Phi(\xi) & = \frac{1}{\psi(\xi) - s^\alpha} - \frac{e^{i \theta}}{\alpha s^{\alpha - 1}} \, \frac{1}{\xi - s e^{i \theta}} + \frac{e^{-i \theta}}{\alpha s^{\alpha - 1}} \, \frac{1}{\xi + s e^{-i \theta}} , \qquad \xi \in \C \setminus i \RR.
\end{aligned}
\]
The second term in the above expression for $K_1(s)$ is given by Lemma~\ref{lm:pv}: it is equal to
\[
\begin{aligned}
 & \frac{i e^{i \theta}}{2 \alpha s^{\alpha - 1}} \bigl(\laplace f_+(-i s e^{i \theta}) - \laplace f_-(-i s e^{i \theta})\bigr) - \frac{i e^{-i \theta}}{2 \alpha s^{\alpha - 1}} \bigl(\laplace f_+(i s e^{-i \theta}) - \laplace f_-(i s e^{-i \theta})\bigr) \\
 & \qquad = \frac{1}{\alpha s^{\alpha - 1}} \bigl(\im(e^{i \theta} \laplace f_-(-i s e^{i \theta})) - \im(e^{i \theta} \laplace f_+(-i s e^{i \theta}))\bigr)\\
 & \qquad = \frac{1}{\alpha s^{\alpha - 1}} \biggl(\int_{-\infty}^0 e^{s x \sin\theta} \sin(s x \cos\theta + \theta) f(x)dx - \int_0^\infty e^{s x \sin\theta} \sin(s x \cos\theta + \theta) f(x)dx\biggr)
\end{aligned}
\]
(in the last step we used the fact that $f_+ = f \chi_{(0, \infty)}$ and $f_- = f \chi_{(-\infty, 0)}$). To identify the first term in the right-hand side of~\eqref{eq:k1:aux}, recall that the function $\Gpos$ was defined so that
\[
 \fourier \Gpos(\xi) = \sin \tfrac{\pi}{\alpha} \biggl(\frac{\alpha}{\psi(\xi)-1} - \frac{1}{e^{-i \theta} \xi - 1} + \frac{1}{e^{i \theta} \xi + 1} \biggr), \qquad \xi \in \RR.
\]
Therefore, for $s > 0$, the Fourier transform of $\Gpos_s(x) = \Gpos(s x)$ is given by
\[
\begin{aligned}
 \fourier G_s^\uparrow(\xi) & = \frac{\sin \tfrac{\pi}{\alpha}}{s} \biggl(\frac{\alpha}{\psi(\xi/s)-1} - \frac{1}{e^{-i \theta} \xi / s - 1} + \frac{1}{e^{i \theta} \xi / s + 1} \biggr) \\
 & = s^{\alpha - 1} \sin \tfrac{\pi}{\alpha} \biggl(\frac{\alpha}{\psi(\xi) - s^\alpha} - \frac{e^{i \theta}}{s^{\alpha - 1}} \, \frac{1}{\xi - s e^{i \theta}} + \frac{e^{i \theta}}{s^{\alpha - 1}} \, \frac{1}{\xi + s e^{-i \theta}} \biggr)
\end{aligned}
\]
for $\xi \in \RR$. It follows that $\Phi(\xi) = \fourier G_s^\uparrow(\xi) / (\alpha s^{\alpha - 1} \sin \tfrac{\pi}{\alpha})$ for $\xi \in \RR$. Recall that $\Phi(\xi)$ is an analytic function in $\C \setminus i \RR$, bounded by $C \min\{1, |\xi|^{-1}\}$ (see the proof of Proposition~\ref{pr1}). Similarly, $\laplace f(-i \xi)$ is an analytic function in $\C \setminus i \RR$, bounded by $C \min\{1, |\xi|^{-1}\}$ in the sector $\{\xi \in \C : |\arg \xi| \le |\theta|\}$ (by Lemma~\ref{lm:7}). It follows that we can deform the contour of integration in the first term of the right-hand side of~\eqref{eq:k1:aux} to $\RR$, which leads to the integral
\[
\begin{aligned}
 \frac{1}{2 \pi} \int_{-\infty}^\infty \Phi(\xi) \laplace f(-i \xi) d\xi & = \frac{1}{2 \pi \alpha s^{\alpha - 1} \sin \tfrac{\pi}{\alpha}} \int_{-\infty}^\infty \fourier f(-\xi) \fourier G_s^\uparrow(\xi) d\xi \\
 & = \frac{1}{\alpha s^{\alpha - 1} \sin \tfrac{\pi}{\alpha}} \int_{-\infty}^\infty f(x) G_s^\uparrow(x) dx ;
\end{aligned}
\]
we used Plancherel's theorem in the last step. The desired expression for $K_1(s)$ follows. The expression for $K_2(s)$ is obtained from that for $K_1(s)$ by considering $f(x) = g(-x)$.
\end{proof}

\begin{lemma}
\label{lm:oddeformuj2}
With the above assumptions and notation, for $s > 0$ we have
\begin{subequations}
\begin{equation}\label{eq:kl0}
 \im \frac{1}{K_0(s) - i L_0(s)} = \frac{\alpha s^{\alpha - 1} \sin^2 \tfrac{\pi}{\alpha}}{\cos \theta} \, ,
\end{equation}
\begin{equation}\label{eq:kl1}
 \im \frac{K_1(s) - i L_1(s)}{K_0(s) - i L_0(s)} = \frac{\sin \tfrac{\pi}{\alpha}}{\cos \theta} \int_{-\infty}^\infty \bigl(\Gpos(s x) + e^{-s x \sin \theta} \cos(s x \cos \theta + \theta + \tfrac{\pi}{\alpha} \sign x)\bigr) f(x) dx,
\end{equation}
\begin{equation}\label{eq:kl2}
 \im \frac{K_2(s) - i L_2(s)}{K_0(s) - i L_0(s)} = \frac{\sin \tfrac{\pi}{\alpha}}{\cos \theta} \int_{-\infty}^\infty \bigl(\Gpos(-s y) + e^{s y \sin \theta} \cos(s y \cos \theta - \theta + \tfrac{\pi}{\alpha} \sign y)\bigr) g(y) dy.
\end{equation}
\end{subequations}
\end{lemma}
\begin{proof}
We fix $s > 0$. For simplicity, in this proof we write $K_j$ and $L_j$ rather than $K_j(s)$ and $L_j(s)$. Recall that $K_0 = -\cot \tfrac{\pi}{\alpha} \cos \theta / (\alpha s^{\alpha - 1})$ and $L_0 = \cos \theta / (\alpha s^{\alpha - 1})$. Thus,
\[
 K_0^2 + L_0^2 = \biggl(\frac{\cos \theta}{\alpha s^{\alpha - 1} \sin \tfrac{\pi}{\alpha}}\biggr)^2 , \qquad \im \frac{1}{K_0 - i L_0} = \frac{L_0}{K_0^2 + L_0^2} = \frac{\alpha s^{\alpha - 1} \sin^2 \tfrac{\pi}{\alpha}}{\cos \theta} \, ;
\]
formula~\eqref{eq:kl0} follows. Furthermore,
\[
 \im \frac{K_1 - i L_1}{K_0 - i L_0} = \frac{K_1 L_0 - K_0 L_1}{K_0^2 + L_0^2} = \frac{\alpha s^{\alpha - 1} \sin \tfrac{\pi}{\alpha}}{\cos \theta} \, (K_1 \sin \tfrac{\pi}{\alpha} + L_1 \cos \tfrac{\pi}{\alpha}) . 
\]
The expressions for $K_1$ is given in Lemma~\ref{lm:oddeformuj}, while $L_1$ is given by
\[
 L_1 = \frac{\re(e^{i \theta} \laplace f(-i s e^{i \theta}))}{\alpha s^{\alpha - 1}} = \frac{1}{\alpha s^{\alpha - 1}} \int_{-\infty}^\infty e^{-s x \sin \theta} \cos(s x \cos \theta + \theta) f(x) dx . 
\]
It follows that
\[
 \begin{aligned} \im \frac{K_1 - i L_1}{K_0 - i L_0} = \frac{\sin \tfrac{\pi}{\alpha}}{\cos \theta} \int_{-\infty}^\infty \bigl(\Gpos(s x) & - e^{-s x \sin \theta} \sin(s x \cos \theta + \theta) \sin \tfrac{\pi}{\alpha} \sign x \\ & \qquad + e^{-s x \sin \theta} \cos(s x \cos \theta + \theta) \cos \tfrac{\pi}{\alpha}\bigr) f(x) dx . \end{aligned} 
\]
This proves~\eqref{eq:kl1}. Formula~\eqref{eq:kl2} follows from~\eqref{eq:kl1} by substituting $f(x) = g(-x)$.
\end{proof}

By combining the above lemmas, we can finally prove Theorem~\ref{thm:1}.

%%%%%%%%%%%%%%%%%%%%%%%%%%%%%%%%%%%%%%%%%%%%%%%%%%%%%%%%%%%%%%%%%
%%%%%%%%%%%%   THEOREM %%%%%%%%%%%%%%%%%%%%%%%%%

\begin{proof}[Proof of the Theorem~\ref{thm:1}]
Fix $t > 0$. By Theorem~\ref{thm:imagine} and substitution $r = s^\alpha$, we have
\[\begin{aligned}
\int_{-\infty}^\infty \int_{-\infty}^\infty f(x)g(y) p^0_t(x, y) dx dy & = -\frac{1}{2\pi} \int_0^\infty e^{-rt}\Im\phi_4(-r)dr \\
& = -\frac{1}{2\pi} \int_0^\infty e^{-s^\alpha t} \alpha s^{\alpha - 1} \Im\phi_4(-s^\alpha) ds .
\end{aligned}\]
Lemma~\ref{lem:imaginaryPart} provides an expression for $\im \phi_4(-s^\alpha)$. Combining it with Lemmas~\ref{lm:abc} and~\ref{lm:oddeformuj2}, we obtain
\[\begin{aligned}
& \int_{-\infty}^\infty \int_{-\infty}^\infty f(x)g(y) p^0_t(x, y) dx dy \\
& \qquad = \int_0^\infty \frac{e^{-s^\alpha t}}{\cos\theta} \biggl( \int_{-\infty}^\infty e^{-sx\sin\theta} \sin(sx\cos\theta)f(x)dx\biggr) \biggl( \int_{-\infty}^\infty e^{sy\sin\theta}\sin(sy\cos\theta)g(y)dy\biggr) ds \\
& \qquad + \int_0^\infty \frac{e^{-s^\alpha t}}{\cos \theta} \biggl( \int_{-\infty}^\infty \bigl(\Gpos(s x) + e^{-s x \sin \theta} \cos(s x \cos \theta + \theta + \tfrac{\pi}{\alpha} \sign x)\bigr) f(x) dx \biggr) \times \\
& \qquad \hspace*{7.5em} \times \biggl(\int_{-\infty}^\infty \bigl(\Gpos(-s y) + e^{s y \sin \theta} \cos(s y \cos \theta - \theta + \tfrac{\pi}{\alpha} \sign y)\bigr) g(y) dy \biggr) ds .
\end{aligned}\]
Thus, the last two integrands in the right-hand side are equal to $-\Fpos(s x) f(x)$ and $-\Fpos(-s y) g(y)$, respectively. Indeed, for $s > 0$ we have
\begin{equation}
\label{eq:FnegFtau}
\begin{split}
 \Fpos(s x) & = e^{-s x \sin\theta} \sin\bigl(s |x| \cos\theta + \theta \sign x + \tfrac{\pi}{\alpha} - \tfrac{\pi}{2} \bigr) - \Gpos(s x) \\
 & = -e^{-s x \sin\theta} \cos\bigl(s |x| \cos\theta + \theta \sign x + \tfrac{\pi}{\alpha} \bigr) - \Gpos(sx) \\
 & = -e^{-s x \sin\theta} \cos\bigl(s x \cos\theta + \theta + \tfrac{\pi}{\alpha} \sign x \bigr) - \Gpos(sx) ,
\end{split}
\end{equation}
and a similar formula holds for $\Fpos(-s y)$. The proof is complete.
\end{proof}

\begin{remark}
For $\theta=0$, i.e. the symmetric case, we reproduce the result of Example~5.1 in~\cite{symzero1}. In this case function $\Gpos=\Gneg$ has explicit form
\begin{equation}
\Gpos(sx)=\frac{\alpha s^{\alpha-1}\sin\frac{\pi\alpha}{2}\sin\frac{\pi}{\alpha}}{\pi}\int_0^\infty \frac{t^\alpha}{1-2t^\alpha\cos(\pi\alpha)+t^{2\alpha}}e^{-s|x|t}dt,
\end{equation}
$\Fpos=\Fneg$ and 
\begin{equation}
\Fpos(sx)=\sin\left( |sx|+\frac{\pi}{\alpha}-\frac{\pi}{2}\right)-\Gpos(sx).
\end{equation}
\end{remark}

\section{Hitting time}
In this section we obtain the formula for $P(\tau^0 > t)$. Our starting point is the expression for the Laplace transform of $\tau^0$:
\[ \int_0^\infty e^{-\lambda t}\PP^x(\tau_0> t)dt = \frac{1-\EE^x e^{-\lambda\tau_0}}{\lambda} = \frac{1}{\lambda} - \frac{u_\lambda(-x)}{\lambda u_\lambda(0)} \]
(see~\eqref{expected}). Our plan is as follows: with the notation of Section~\ref{sec:contour}, we define
\[
 \phi_5(\lambda) = \frac{\laplace g(0)}{\lambda} - \frac{\phi_2(\lambda)}{\lambda \phi_0(\lambda)}
\]
when $\lambda \in \C \setminus (-\infty, 0]$, and, as usual, we let $\phi_5(-\lambda) = \lim_{\eps \to 0^+} \phi_5(-\lambda + i \eps)$ when $\lambda > 0$. First, we will show that whenever $g \in \Ha$, we have
\begin{equation}
\label{eq:tau:goal}
\begin{aligned}
\int_0^\infty \int_\RR e^{-\lambda t}\PP^x(\tau_0 > t) g(x) dx dt & = -\frac{1}{\pi}\int_0^\infty \int_0^\infty e^{-\lambda t}e^{-rt}\Im \phi_5(-r)drdt \\ & \hspace*{-3em} = -\frac{1}{\pi\cos\theta}\int_0^\infty \int_0^\infty e^{-\lambda t}e^{-s^\alpha t}\int_\RR \frac{\Fneg(sx)}{s}g(x)dx ds dt .
\end{aligned}
\end{equation}
Then, we will change the order of integration. The desired result will then follow by a density-type argument: the class of admissible functions $g$ is dense in an appropriate sense, and both sides of~\eqref{formula:1} are continuous functions of $x$. Note, however, that changing the order of integration is not merely an application of Fubini's theorem: the integral in the right-hand side of~\eqref{eq:tau:goal} is \emph{not} absolutely convergent. For this reason, we will first deform the contour of integration, only then apply Fubini's theorem, and then deform the contour back.

\begin{lemma}\label{lm:hit:sqint}
The function $\lambda \phi_5(-\lambda^2)$ is in the Nevanlinna class $\mathcal{N}^+$, and it is in $\leb^p(\RR)$ for some $p > 1$.
\end{lemma}

\begin{proof}
The proof of the first part of the lemma is a minor modification of the proof of Lemma~\ref{lm:bounded.type}: the functions $\lambda$ and $\lambda \phi_0(-\lambda^2)$ are outer functions, $\lambda \phi_2(-\lambda^2)$ is in $\mathcal{N}^+$, and therefore $\lambda \phi_5(-\lambda^2)$ is in $\mathcal{N}^+$. In order to prove the other part of the lemma, we need to show that
\begin{equation}
\int_\RR |\lambda \phi_5(-\lambda^2)|^2 d\lambda < \infty.
\end{equation}
By Sokhozki formula (see the proof of Lemma~\ref{lem:imaginaryPart}), we have
\[
 \begin{split}
 \phi_5(-s^\alpha) & = \frac{\laplace g(0)}{-s^\alpha} - \frac{K_2(s)-iL_2(s)}{-s^\alpha (K_0(s)-iL_0(s))} \\
 & = \frac{K_2(s)-iL_2(s) - (K_0(s) - iL_0(s)) \laplace g(0)}{s^\alpha (K_0(s)-iL_0(s))}.
 \end{split}
\]
By the definition of $L_2(s)$, we have
\[
 L_2(s) - L_0(s) \laplace g(0) = \frac{\re(e^{i \theta} \laplace g(i s e^{i \theta})) - \laplace g(0) \cos \theta}{\alpha s^{\alpha - 1}} = \frac{\re(e^{i \theta} (\laplace g(i s e^{i \theta}) - \laplace g(0)))}{\alpha s^{\alpha - 1}} ,
\]
and hence
\[
 |L_2(s) - L_0(s) \laplace g(0)| \le \frac{|\laplace g(i s e^{i \theta}) - \laplace g(0)|}{\alpha s^{\alpha - 1}} .
\]
Using the bounds $|\laplace g(z)| \le C \min\{1, |z|^{-1}\}$ and $|\laplace g'(z)| \le C \min\{1, |z|^{-2}\}$ (see Lemmas~\ref{lm:7} and~\ref{lm:7p}), we arrive at
\[
 |L_2(s) - L_0(s) \laplace g(0)| \le \frac{C \min\{s, 1\}}{\alpha s^{\alpha - 1}} .
\]
In a very similar way, by the definition of $K_2(s)$,
\[
 \begin{split}
 K_2(s) - K_0(s) \laplace g(0) & = \frac{1}{\pi} \pvint_0^\infty \frac{\re(e^{i \theta} \laplace g(i r e^{i \theta})) - \laplace g(0) \cos \theta}{r^\alpha - s^\alpha} \, dr \\
 & = \frac{1}{\pi} \pvint_0^\infty \frac{\re(e^{i \theta} (\laplace g(i r e^{i \theta}) - \laplace g(0)))}{r^\alpha - s^\alpha} \, dr ,
 \end{split}
\]
and by Lemmas~\ref{lm:7} and~\ref{lm:7p}, the function $h(r) = \laplace g(i r e^{i \theta}) - \laplace g(0)$ satisfies $|h(r)| \le C \min\{r, 1\}$ and $|h'(r)| \le C \min\{1, r^{-2}\}$. Repeating the proof of Lemma~\ref{lm:square:Kfg0}, with appropriately modified estimates~\eqref{eq:Kfg0:1}, \eqref{eq:Kfg0:2} and~\eqref{eq:Kfg0:3}, we find that %TODO: to powinno być częścią lematu lm:square:Kfg0...
\[
 |K_2(s) - K_0(s) \laplace g(0)| \le \frac{C_1 \log(1 + 2 s)}{s^{\alpha - 1} (1 + s)} + C_2 s^{1 - \alpha} \min\{s, 1\} + C_3 \min\{s^{1 - \alpha}, 1\} ,
\]
and therefore
\[
 |K_2(s) - K_0(s) \laplace g(0)| \le C \min\{1, s^{1 - \alpha}\} .
\]
The above bounds and the definitions of $K_0(s)$ and $L_0(s)$ imply that
\[
 |\phi_5(-s^\alpha)| \le \frac{C \min\{1, s^{1 - \alpha}\}}{s^\alpha s^{1 - \alpha}} = C \min\{s^{-1}, s^{-\alpha}\} .
\]
We conclude that
\[
 |\lambda \phi_5(-\lambda^2)| \le C \min\{\lambda^{1 - 2 / \alpha}, \lambda^{-1}\} ,
\]
so that $\lambda \phi_5(-\lambda^2)$ is in $\leb^p(\RR)$ for every $p \in (1, \tfrac{\alpha}{2 - \alpha})$.
\end{proof}

\begin{lemma}
We have, for $\lambda>0$, 
\label{lm:prob.cauchy}
\begin{equation}\label{eq:prob:cauchy}
\phi_5(\lambda)=-\frac{1}{\pi}\int_0^\infty \int_0^\infty e^{-\lambda t}e^{-st}\Im \phi_5(-s)dsdt.
\end{equation}
\end{lemma}
\begin{proof}
As in the proof of Theorem~\ref{thm:imagine}, we find that, by Lemma~\ref{lm:hit:sqint}, the function $\sqrt{\lambda} \phi_5(\lambda)$ satisfies the assumptions of Corollary~\ref{cor:branges}. It follows that for $\lambda \in \C \setminus (-\infty, 0]$,
\[
\sqrt{\lambda} \phi_5(\lambda) = -\frac{\sqrt{\lambda}}{\pi} \int_0^\infty \frac{\im \phi_5(-r)}{r + \lambda} \, dr ,
\]
which implies~\eqref{eq:prob:cauchy}.
\end{proof}

\begin{proof}[Proof of Theorem~\ref{thm:3}]
Let $g\in\Ha$. Recall that for $t > 0$,
\[
 \int_0^\infty e^{-\lambda t}\PP^x(\tau_0> t)dt = \frac{1}{\lambda} - \frac{u_\lambda(-x)}{\lambda u_\lambda(0)} .
\]
Since $g$ is integrable, by Fubini's theorem, for every $\lambda > 0$ we have
\begin{equation}
\int_0^\infty \int_\RR e^{-\lambda t} \PP^x(\tau_0> t) g(x)dxdt= \frac{\laplace g(0)}{\lambda} - \frac{1}{\lambda u_\lambda(0)} \int_\RR u_\lambda(-x) g(x)dx.
\end{equation}
By Plancherel's theorem and an argument used in Lemma~\ref{lem:fgp2},
\begin{equation}
 \int_\RR u_\lambda(-x)g(x)dx = \frac{1}{2 \pi} \int_\RR \frac{\laplace g(i\xi)}{\psi(\xi)+\lambda}d\xi=\phi_2(\lambda) ,
\end{equation}
and similarly $u_\lambda(0) = \phi_0(\lambda)$. Therefore,
\begin{equation}
\int_0^\infty \int_\RR e^{-\lambda t} \PP^x(\tau_0> t) g(x)dxdt= \frac{\laplace g(0)}{\lambda} - \frac{\phi_2(\lambda)}{\lambda \phi_0(\lambda)} = \phi_5(\lambda) .
\end{equation}
By Lemma~\ref{lm:prob.cauchy} we obtain
\begin{equation}
\begin{split}
\int_0^\infty \int_\RR e^{-\lambda t} \PP^x(\tau_0> t) g(x)dxdt & = -\frac{1}{\pi}\int_0^\infty \int_0^\infty e^{-\lambda t}e^{-rt}\Im \phi_5(-r)drdt \\
& = \frac{1}{\pi}\int_0^\infty \int_0^\infty e^{-\lambda t}e^{-s^\alpha t}\alpha s^{\alpha-1}\Im \frac{\phi_2(-s^\alpha)}{-s^\alpha\phi_0(-s^\alpha)}dsdt
\end{split}
\end{equation}
for every $\lambda > 0$. From the uniqueness of the Laplace transform we get, for almost all $t > 0$,
\begin{equation}
\int_\RR \PP^x(\tau_0> t) g(x)dx = - \frac{\alpha}{\pi} \int_0^\infty \frac{e^{-s^\alpha t}}{s} \im \frac{\phi_2(-s^\alpha)}{\phi_0(-s^\alpha)} ds.
\end{equation}
By Lemma~\eqref{lm:oddeformuj2}, equality~\eqref{eq:FnegFtau} and the fact that $\phi_j(-s^\alpha)= K_j(s) - i L_j(s)$, $j=0,2$, we get
\begin{equation}
\label{eq:hit:exp1}
\begin{split}
\im \frac{\phi_2(-s^\alpha)}{\phi_0(-s^\alpha)} & = \Im \frac{K_2(s) - i L_2(s)}{K_0(s) - i L_0(s)} \\
 & = \frac{\sin \tfrac{\pi}{\alpha}}{\cos \theta} \int_{-\infty}^\infty (\Gneg(s x) + e^{s x \sin \theta} \cos(s x \cos \theta - \theta + \tfrac{\pi}{\alpha} \sign x)) g(x) dx \\
 & = - \frac{\sin \tfrac{\pi}{\alpha}}{\cos\theta} \int_\RR \Fneg(sx) g(x)dx .
\end{split}
\end{equation}
We have thus proved that for almost all $t > 0$,
\begin{equation}
\label{prob:contour}
\int_\RR \PP^x(\tau_0> t) g(x)dx = \frac{\alpha \sin \tfrac{\pi}{\alpha}}{\pi \cos\theta} \int_0^\infty \frac{e^{-s^\alpha t}}{s} \int_\RR \Fneg(sx) g(x)dx ds.
\end{equation}
We now change the order of integration in the right-hand side of~\eqref{prob:contour}. With no loss of generality, we assume that $x \sin \theta \ge 0$; the other case is dealt with in a similar manner. Recall that 
\[
 \Fneg(sx)= e^{sx \sin\theta}\sin\left(s|x|\cos\theta-\theta\sign x +\frac{\pi}{\alpha}-\frac{\pi}{2}\right)-\Gneg(sx).
\]
We split the integral in the right-hand side of~\eqref{prob:contour} into three parts:
\[
 \int_0^\infty e^{-s^\alpha t}\int_\RR \frac{\Fneg(sx)}{s}g(x)dx ds = I_1 + I_2 + I_3,
\]
where
\[
\begin{aligned}
 I_1 & = \int_0^\infty e^{-s^\alpha t}\int_{-\infty}^0 \frac{\Fneg(sx)}{s}g(x)dx ds , \\
 I_2 & = \int_0^\infty e^{-s^\alpha t}\int_0^\infty \frac{e^{-s x} \sin \varphi - \Gneg(sx)}{s}g(x)dx ds , \\
 I_3 & = \int_0^\infty e^{-s^\alpha t}\int_0^\infty \frac{e^{sx \sin\theta}\sin(s|x|\cos\theta + \varphi) - e^{-s x} \sin \varphi}{s}g(x)dx ds ,
\end{aligned}
\]
with the notation $\varphi = \tfrac{\pi}{\alpha} - \tfrac{\pi}{2} - \theta$. Recall that by Lemma~\ref{lm:Hold}, the function $\Gneg$ is H\"older continuous on $(-\infty, 0)$ with exponent $\alpha - 1$, and $\Gneg(0^-) = \sin(\tfrac{\pi}{\alpha} - \tfrac{\pi}{2} + \theta)$. It follows that $\Fneg$ is H\"older continuous on $(-\infty, 0]$ with exponent $\alpha - 1$, and $\Fneg(0) = 0$. Furthermore, $\Fneg$ is bounded on $(-\infty, 0]$, and hence $|\Fneg(s x)| \le C \min\{1, (s |x|)^{\alpha - 1}\}$ for $x \in (-\infty, 0]$. Finally, by the definition of $\Ha$, we have $|g(x)| \le C \min\{1, |x|^{-\delta |x|}\}$. We conclude that the integrand in the double integral in the definition of $I_1$ is bounded by $C \min\{1, (s |x|)^{\alpha - 1}\} s^{-1} \min\{1, |x|^{-\delta x}\} e^{-s^\alpha t}$, and
\[
\begin{aligned}
 & \int_{-\infty}^0 \int_0^\infty \min\{1, (s |x|)^{\alpha - 1}\} s^{-1} \min\{1, |x|^{-\delta x}\} e^{-s^\alpha t} ds dx \\
 & \qquad \le \int_{-\infty}^0 \biggl(\int_0^{1 / |x|} (s |x|)^{\alpha - 1} s^{-1} ds + \int_{1 / |x|}^\infty s^{-1} e^{-s^\alpha t} ds\biggr) \min\{1, |x|^{-\delta x}\} dx \\
 & \qquad \le \int_{-\infty}^0 (C(t) + \log (1 + |x|^{-1})) \min\{1, |x|^{-\delta x}\} dx < \infty .
\end{aligned}
\]
Thus, the integral in the definition $I_1$ converges absolutely, and by Fubini's theorem,
\[
 I_1 = \int_{-\infty}^0 g(x) \int_0^\infty e^{-s^\alpha t} \frac{\Fneg(sx)}{s} ds dx .
\]
A similar argument applies to $I_2$: again by Lemma~\ref{lm:Hold}, $e^{-s x} \sin \varphi - \Gneg(sx)$ is bounded and H\"older continuous on $(0, \infty)$ with right limit at $0$ equal to zero, so that we may use Fubini's theorem. It follows that
\[
 I_2 = \int_0^\infty g(x) \int_0^\infty e^{-s^\alpha t} \frac{e^{-s x} \sin \varphi - \Gneg(sx)}{s} ds dx .
\]
The integral $I_3$, however, requires a more subtle treatment. We split it further into two parts, which are dealt with in a very similar way: since
\[
 e^{s x \sin \theta} \sin(sx\cos\theta+\varphi) = \frac{e^{i(sxe^{-i \theta}+\varphi)} - e^{-i(sxe^{i \theta}+\varphi)}}{2i} ,
\]
we have $I_3 = (I_4 - I_5) / (2 i)$, where
\[
\begin{aligned}
 I_4 & = e^{i \varphi} \int_0^\infty e^{-s^\alpha t}\int_0^\infty \frac{e^{i s x e^{-i \theta}} - e^{-s x}}{s}g(x)dx ds , \\
 I_5 & = e^{-i \varphi} \int_0^\infty e^{-s^\alpha t}\int_0^\infty \frac{e^{-i s x e^{i \theta}} - e^{-s x}}{s}g(x)dx ds .
\end{aligned}
\]
Recall that, by the definition of $\Ha$, $g$ extends to an analytic function in $\CC\setminus i\RR$, which is bounded by $C \min \{ 1, |x|^{-\delta |x|}\}$ in the sector $\{ x\in \CC: |\arg x|\leqslant |\theta|\}$. Furthermore, $\exp(isxe^{-i\theta}) - \exp(-s x)$ is an entire function of $x$, bounded by $2 e^{s|x|}$. Hence we may deform the contour of integration in the inner integral from $(0, \infty)$ to $(0, e^{i \theta} \infty)$, and find that
\begin{equation}
\label{eq:i4}
\begin{split}
 I_4 & = e^{i \varphi} \int_0^\infty e^{-s^\alpha t }\int_{(0, e^{i \theta} \infty)} \frac{e^{i s x e^{-i \theta}} - e^{-s x}}{s}g(x)dx ds \\
 & = e^{i \varphi} \int_0^\infty e^{-s^\alpha t}\int_0^\infty \frac{e^{i r s} - e^{-r s e^{i \theta}}}{s}g(e^{i \theta} r) e^{i \theta} dr ds .
\end{split}
\end{equation}
The exponential function is Lipschitz continuous in the right complex half-plane, with Lipschitz constant $1$. Therefore, $|e^{i r s + i \varphi} - e^{-r s e^{i \theta} + i \varphi}| \le \min\{2, 2 r s\}$ when $r, s > 0$. By the argument used in the analysis of $I_1$, it follows that the double integral in~\eqref{eq:i4} converges absolutely, and so, by Fubini's theorem,
\[
\begin{split}
 I_4 & = e^{i \varphi} \int_0^\infty g(e^{i \theta} r) e^{i \theta} \int_0^\infty e^{-s^\alpha t} \frac{e^{i r s} - e^{-r s e^{i \theta}}}{s} ds dr \\
 & = e^{i \varphi} \int_{(0, e^{i \theta} \infty)} g(x) \int_0^\infty e^{-s^\alpha t} \frac{e^{i s x e^{-i \theta}} - e^{-s x}}{s} ds dx .
\end{split}
\]
By Lemma~\ref{lem:aux:int} and the inequality $\tfrac{\pi}{2} + |\theta| \le \tfrac{\pi}{\alpha} < \tfrac{\pi}{2} + \tfrac{\pi}{2 \alpha}$, the inner integral is a bounded analytic function in the sector $|\arg x| \le \theta$. Since $|g(x)| \le C \min\{1, |x|^{-\delta |x|}\}$ in this sector, we may deform the contour of integration, and eventually find that
\[
 I_4 = e^{i \varphi} \int_0^\infty g(x) \int_0^\infty e^{-s^\alpha t} \frac{e^{i s x e^{-i \theta}} - e^{-s x}}{s} ds dx ,
\]
which is identical to the definition of $I_4$, except that the integrals are in reverse order.

A very similar argument shows that the order of integration can be reversed in the definition of $I_5$, and thus also in $I_3$. We conclude that
\[
 \int_0^\infty e^{-s^\alpha t}\int_\RR \frac{\Fneg(sx)}{s}g(x)dx ds = I_1 + I_2 + I_3 = \int_\RR g(x) \int_0^\infty e^{-s^\alpha t} \frac{\Fneg(sx)}{s} ds dx ,
\]
By~\eqref{prob:contour}, for almost every $t > 0$ and every $g \in \Ha$,
\[
 \int_\RR g(x) \PP^x(\tau_0 > t) dx = \frac{\alpha \sin \tfrac{\pi}{\alpha}}{\pi \cos \theta} \int_\RR g(x) \int_0^\infty \frac{e^{-s^\alpha t}}{s} \, \Fneg(sx) ds dx .
\]
By Lemma~\ref{lm:density}, we have
\[
 \PP^x(\tau_0 > t) = \frac{1}{\pi \cos \theta} \int_0^\infty \frac{e^{-s^\alpha t}}{s} \, \Fneg(sx) ds
\]
for almost all $x \in \RR \setminus \{0\}$ and $t > 0$. Since both sides are jointly continuous functions of $x \in \RR \setminus \{0\}$ and $t > 0$ (the right-hand side by a simple application of Lebesgue's dominated convergence theorem), the above equality in fact holds for all $x \in \RR \setminus \{0\}$ and $t > 0$, and the proof is complete.
\end{proof}

\section*{Acknowledgement}
I  express my gratitude to Professor Mateusz Kwaśnicki for his guidance and numerous comments to the preliminary version of this article.

\end{document}